\documentclass[11pt]{article}

\usepackage{amsmath}
\usepackage{amssymb}
\usepackage{amsthm}
\usepackage{latexsym}
\usepackage{eurosym}
\usepackage{color}
\usepackage{float}
\usepackage{graphicx}
\usepackage{color}
\usepackage{epstopdf}
\usepackage{enumerate}
\usepackage[T1]{fontenc}
\usepackage[english]{babel}
\usepackage[table]{xcolor}

\DeclareSymbolFont{calletters}{OMS}{cmsy}{m}{n}
\DeclareSymbolFontAlphabet{\mathcal}{calletters}
\def\be{\begin{eqnarray}}
\def\ee{\end{eqnarray}}

\def\b*{\begin{eqnarray*}}
\def\e*{\end{eqnarray*}}

\newtheorem{Theorem}{Theorem}[part]

\makeatletter \@addtoreset{equation}{section}

\@addtoreset{Definition}{section}

\@addtoreset{Theorem}{section}

\@addtoreset{Proposition}{section}

\@addtoreset{Property}{section}

\@addtoreset{Assumption}{section}

\@addtoreset{Corollary}{section}

\@addtoreset{Lemma}{section}

\@addtoreset{Remark}{section}

\@addtoreset{Example}{section}

\@addtoreset{Condition}{section}

\def\={\;=\;}
\def\.{\;.}

\def\1{{\bf 1}}

\def\b*{\begin{eqnarray*}}
\def\e*{\end{eqnarray*}}

 \def\normeL2#1{\left\|{#1}\right\|_{L^2}}

\begin{document}

%\title{Optimal Real-Time Bidding Strategies and the Pricing of Ad-Buying Services\footnote{This research has been conducted with the support of the Research Initiative ``\'Economie et gestion des nouvelles donn\'ees''.}}
\title{Optimal Real-Time Bidding Strategies\footnote{The authors would like to thank Dominique Delport (Havas Media), Julien Laugel (MFG Labs), Pierre-Louis Lions (Coll\`ege de France), and Arnaud Parent (Havas Media) for the discussions they had on the topic. Two anonymous referees also deserve to be thanked.}}

\author{Joaquin Fernandez-Tapia\footnote{Universit\'e Pierre et Marie Curie, Laboratoire de Probabilit\'es et Mod\`eles Al\'eatoires (LPMA). 4 Place Jussieu, 75005 Paris. Joaquin acknowledges support in 2015 from the Chair ``\'Economie et gestion des nouvelles donn\'ees''.}, Olivier Gu\'eant\footnote{Ensae-Crest, 3 avenue Pierre Larousse, 92245 Malakoff Cedex, France. Member of the Scientific Advisory Board of Havas Media. The content of this article does not reflect the official opinion or the practices of Havas Media. Corresponding author: \texttt{olivier.gueant@ensae.fr}.}, Jean-Michel Lasry\footnote{Institut Louis Bachelier. Member of the Scientific Advisory Board of Havas Media. The content of this article does not reflect the official opinion or the practices of Havas Media.}}

\date{}

\maketitle
\begin{abstract}

The ad-trading desks of media-buying agencies are increasingly relying on complex algorithms for purchasing advertising inventory. In particular, Real-Time Bidding (RTB) algorithms respond to many auctions -- usually Vickrey auctions -- throughout the day for buying ad-inventory with the aim of maximizing one or several key performance indicators (KPI). The optimization problems faced by companies building bidding strategies are new and interesting for the community of applied mathematicians. In this article, we introduce a stochastic optimal control model that addresses the question of the optimal bidding strategy in various realistic contexts: the maximization of the inventory bought with a given amount of cash in the framework of audience strategies, the maximization of the number of conversions/acquisitions with a given amount of cash, etc. In our model, the sequence of auctions is modeled by a Poisson process and the \textit{price to beat} for each auction is modeled by a random variable following almost any probability distribution. We show that the optimal bids are characterized by a Hamilton-Jacobi-Bellman equation, and that almost-closed-form solutions can be found by using a fluid limit. Numerical examples are also provided.

\vspace{0.5cm}

\vspace{5mm}

\end{abstract}

\section{Introduction}

From the viewpoint of a company launching an advertising campaign, the goal of digital advertising is to increase its return on investment by leveraging the different channels enabling an interaction with its potential customers: desktop display, mobile, social media, e-mailing, etc. Usually, this is achieved via branding campaigns, by prospecting individuals who are likely to be in affinity with a given product/campaign, or by driving those who have already shown some interest into a final conversion (\emph{e.g.} a purchase).\\

In recent years, the advertising industry has gone through a lot of upheavals: numerous technological changes, a deluge of newly available data, the emergence of a huge number of ad-tech startups entering the market, etc. In particular, new mechanisms have emerged and have completely changed the way digital ad inventory is purchased. In practice, the inventory is often purchased programmatically, and it is possible to algorithmically buy it unit by unit, with the hope of making real the original promise of the advertising and media buying industry: \textit{targeting the right person, at the right time, and in the right context.}\\

Programmatic media buying has skyrocketed over the last five years. Although these figures can only be rough approximations, it is estimated that the total net advertising revenue linked to programmatic desktop display in Europe was around \EUR{2.9}bn in 2014. For programmatic mobile display and video display the figures were respectively \EUR{552}m and \EUR{205}m -- see \cite{iab}. Overall, the total growth in net advertising revenue related to programmatic media buying was around 70\% in Europe between 2013 and 2014. IAB Europe estimates in \cite{iab} that the percentage of revenue coming from programmatic media buying is a two-digit number for all formats: 39\% for desktop display, 27\% for mobile display, and 12\% for video display. In the US, the figures are even more staggering with \$5.89bn spent programmatically on desktop/laptop display, and \$4.44bn on mobile/tablet display, in 2014 (source: eMarketer.com).\\

One of the main and most exciting developments in programmatic media buying is Real-Time Bidding (or RTB). RTB is a new paradigm in the way digital inventory is purchased: advertisers\footnote{Mostly through intermediaries such as media agencies, but sometimes also by themselves when they set up trading desks.} can buy online inventory through real-time auctions for displaying a banner (or a short video). These real-time auctions make it possible for advertisers to target individual users on a per-access basis.\\

In a nutshell, each time a user visits a website, the publisher -- the supply side -- connects to a virtual marketplace, called an \textit{ad exchange}, in order to trigger an auction for each available slot that can be allocated to advertising. On the demand side, ad trading desks receive auction requests (sometimes through a Demand-Side Platform -- DSP), together with information about the user, the type of website, etc., and choose the bid level that best suits their strategy. Once the different bids are processed, the slot is attributed to the bidder who has proposed the highest bid and the price paid depends on the type of auction. The entire process, from the user entering the website to the display of the banner, takes around 100 milliseconds.\\

RTB auctions are usually of the Vickrey type, also known as ``second-price auctions''. In short, the mechanism is the following: first, the participants send  their bids in a sealed way, then, the item (here the slot) is sold to the participant who has proposed the highest bid, and the price paid by this participant corresponds to the second best bid (or to a minimum price if there is only one participant). Structurally, Vickrey auctions give participants an incentive to reveal their true valuation for the item -- see \cite{vickrey}.\\

The problem faced by ad-trading desks is to choose the optimal bid level each time they receive a request to participate in a Vickrey auction. Here, optimality may have different meanings, depending on the considered key performance indicator (KPI). In all cases, the complexity of the problem arises from the need of optimizing a functional depending on \textit{macroscopic} quantities at an hourly, daily or weekly timescale, by interacting with the system at the \textit{microscopic} scale of each auction, \emph{i.e.} through a high-frequency/low-latency bidding algorithm participating in thousands of auctions per second.
This multi-scale feature leads to the need of mathematical models that are both realistic and tractable, because numerical methods are often cumbersome and time-consuming in the case of multi-scale problems. In this article, we rely on methods coming from stochastic optimal control and we show that the optimal bidding strategy can be approximated very precisely (and almost in closed form) by using classical tools of convex optimization.\\

Besides the classical literature on Vickrey auctions (see for instance \cite{vickrey2,vickrey3,vickrey}) -- which is related to auction theory and more generally to game theory --, the academic literature on this new kind of problems is really scarce. General approaches for Real-Time Bidding optimization from a buyer's perspective can be found mostly in conference proceedings from the computer-science community (\emph{e.g.} \cite{lee,zhang}). Our approach is similar to the one presented in the work of Amin \emph{et al.} \cite{amin}: both are Markov Decision Process (MDP) approaches\footnote{Even though ours is in continuous time.} and dealing with similar auction problems. However, besides the originality of their model, Amin \emph{et al.} do not extend their mathematical development beyond the baseline discrete case. Another author introduced an MDP approach in the conference paper \cite{yuan}, but he focused on the problem from a publisher perspective. In general, the supply-side perspective has generated more academic research than the demand-side one (see Yuan's PhD dissertation \cite{yuan2} and the articles by Balseiro \emph{et al.} \cite{balseiro1,balseiro2}). A recent study of RTB auctions from a buyer's perspective is Stavrogianni's PhD dissertation \cite{stravrogiannis}.\\

Our stochastic optimal control approach is inspired by the academic literature in algorithmic trading \cite{hft1,hft2,hft3}, where, similarly to our problem, the goal is to optimize a macroscopic functional depending on the terminal state of the algorithm (\emph{e.g.} at the end of the day) by continuously making decisions on a high-frequency basis (\emph{i.e.} milliseconds). Moreover, like in high-frequency trading models involving limit orders, the algorithm should react to a system driven by one or several controlled Poisson processes.\\

In this paper -- the first of a series on Real-Time Bidding, see \cite{fglpricing} and \cite{fgllearning} --, we model by a marked Poisson process the sequence of auction requests received by an ad-trading desk: the Poisson process models the arrival times of the requests, and the marks correspond to independent random variables $(p_n)_{n \in \mathbb{N}^*}$ modeling the \textit{price to beat}, \emph{i.e.} the highest bid proposed by the other participants' in the auction.\footnote{In particular, by assuming that the best price of other participants is an exogenous variable, we ignore all game-theoretical effects, or, equivalently, all strategic interactions between market participants. Although we do not consider these aspects in the present paper, it would be very interesting to consider several agents -- or continuums of agents in a mean field game setup -- with different goals and to look for the Nash equilibriums in the repeated auction game at the heart of our model. This would be particularly relevant for auctions on small and specific parts of the audience.} Every time an auction is received, the algorithm sends a bid $b$ to the auction server. For the $n^{\text{th}}$ auction, the inventory is purchased by the algorithm if and only if the bid sent by the algorithm is greater than the price to beat $p_n$ (and in that case the price paid for the slot is $p_n$). The rationale for considering this \textit{statistical model}, rather than a more complicated game-theoretical one, comes from: (i) the large number of auction requests (several hundreds per second) for most segments of audience, and (ii) our assumption that the algorithm is restricted to an homogeneous subset of the inventory (\emph{i.e.} we assume that a segmentation of the different audiences and contexts has been carried out beforehand, or, in other words, that the problem we consider is at the tactical ``execution'' level -- see also \cite{f1,f2}).\\

In Section 2, we introduce the main notations of our modeling framework, and we focus on a stochastic optimal control problem where an ad trader aims at maximizing the total number of banners displayed, for a given spending (audience strategy). We exhibit the characterization of the optimal bidding strategy with a Hamilton-Jacobi-Bellman equation, and show that the optimal bidding strategy can be obtained in almost-closed form by using a fluid-limit approximation. Moreover, this approximation leads to a new characterization of the bidding strategy in the form of an optimal scheduling, which in practice can be tracked by a feedback-control mechanism. The main result is indeed that the budget should be spent evenly over the considered time window.\\

In Section 3, we propose several extensions of our model. In particular, we generalize the initial model by considering several sources and types of inventory, and we consider another objective function for taking account of a very important KPI: the number of conversions. In these extensions, as in the initial model with only one source of inventory, we find that the total budget should be spent evenly over the considered time window -- see \cite{fgllearning} for a different conclusion when on-line learning is considered. However, we find that the optimal bidding strategies are not uniform across sources, but instead proportional to an index which is, in the general case with multiple sources and conversions,  a simple function of (i) the \emph{a priori} interest for each source/type of inventory, (ii) the \emph{a priori} interest for each source of conversion, and (iii) the probability of conversion associated to each source. In Section 3, we also discuss the difference between first-price and second-price auctions, the impact of floor prices, and more general -- nonlinear -- objective criterions.

\section{A model for audience strategies}

In this section, we present a model where an ad trader wishes to spend a given amount of money $\bar{S}$ over a time window $[0,T]$, in order to display a maximum number of banners to a given population. He receives auction requests at random times from a single ad exchange -- and does not know in advance how many auction requests he will receive. The model we propose is naturally written in continuous time and the (random) occurrences of auction requests are modeled by the jumps of a Poisson process. A simplified discrete-time modeling framework is presented in the appendix for readers who are more used to discrete-time Markov decision processes.\\

\subsection{The modeling framework in continuous time}

Let us fix a probability space $(\Omega, \mathcal{F}, \mathbb{P})$ equipped with a filtration $(\mathcal{F}_t)_{t\in \mathbb{R}_+}$ satisfying
the usual conditions. We assume that all stochastic processes are defined on $(\Omega, \mathcal{F},(\mathcal{F}_t)_{t\in \mathbb{R}_+}, \mathbb{P})$.\\

\emph{Auctions:}\\

We consider an ad trader connected to an ad exchange.\footnote{The ad trader may be connected to several ad exchanges, but we assume in this first model that he does not make any difference between auction requests coming from the different ad exchanges.} He receives requests to participate in auctions in order to purchase inventory and display some banners to the specific population he wants to target. Requests are modeled with a marked Poisson process: the arrival of requests is triggered by the jumps of a Poisson process $(N_t)_t$ with intensity $\lambda > 0$,\footnote{In this paper, we only consider the case of a constant intensity $\lambda$. A time-varying auction-request intensity can be handled via a change of time, similar as in \cite{f2}, where the main idea is to switch from \textit{physical time} to \textit{trading time}.} and the marks $(p_n)_{n \in \mathbb{N}^*}$ correspond, for each auction, to the highest bid among the other participants' bids.\\

Every time he receives a request to participate in an auction, the ad trader can bid a price: at time $t$, if he receives a request, then we denote his bid by $b_t$. We assume that the ad trader stands ready to bid (possibly a bid equal to $0$ or $+\infty$) at all times. In particular, we assume that $(b_t)_t$ is a predictable process with values in $\mathbb{R}_+ \cup \{+\infty\}$.\\

If at time $t$ the $n^{\text{th}}$ auction occurs, the outcome of this auction is the following:
\begin{itemize}
  \item If $b_t > p_{n}$, then the ad trader wins the auction: he pays the price $p_{n}$ and his banner is displayed.
  \item If $b_t \le p_{n}$, then\footnote{The equality case is not considered separately because we assume that the distribution of the variables $(p_n)_n$ is absolutely continuous -- see below.} the ad trader does not win the auction. In particular, the trader's banner is not displayed and he pays nothing.
\end{itemize}

An important assumption of our model is that $(p_n)_{n \in \mathbb{N}^*}$ are \emph{i.i.d.} random variables distributed according to an absolutely continuous distribution. We denote by $F$ the associated cumulative distribution function and by $f$ the associated probability density function. Our assumptions are the following:
\begin{itemize}
  \item $\forall n\in \mathbb N^*$, $p_n$ is almost surely positive. In particular, $F(0) = 0$.
  %\item $\forall n\in \mathbb N^*$, $p_n \in L^2(\Omega)$.
  \item $\forall p > 0, f(p) > 0$.\footnote{The model can easily be (slightly) modified to cover the case of truncated functions $f$, \emph{i.e.} $f$ equal to 0 above a given price level.}
  \item $\lim_{p \to +\infty} p^3 f(p) = 0$.\\
\end{itemize}

\emph{Remaining cash process:}\\

We denote by $(S_t)_t$ the process modeling the amount of cash to be spent. Its dynamics is:\footnote{This process is well defined because $\forall n\in \mathbb{N}^*, p_n \in L^1(\Omega)$.}
\begin{equation}
\label{dynS}
dS_t = - p_{N_t} \mathbf{1}_{\{b_t> p_{N_t}\}}dN_t, \quad S_0 = \bar{S}.
\end{equation}

\emph{Inventory process:}\\

The number of impressions, \emph{i.e.} the number of banners displayed, is modeled by the inventory process $(I_t)_t$. Its dynamics is:
\begin{equation*}
dI_t = \mathbf{1}_{\{b_t> p_{N_t}\}}dN_t, \quad I_0 = 0.
\end{equation*}

\emph{Stochastic optimal control problem:}\\

In the model of this section, the trader aims at maximizing the expected number of banners displayed over $[0,T]$. We should impose not to spend more than $\bar{S}$, but, for technical reasons, we do prefer to consider a relaxed problem, and to penalize extra-spending. We consider therefore the maximization of the following criterion:
\begin{equation*}
\mathbb E \left[I_T - K \min\left(S_T,0\right)^2\right],
\end{equation*}
where $K$ measures the importance of the penalty for extra-spending.\footnote{In particular, we will consider the limit case $K \to +\infty$ as a way to see what happens when the maximal total amount to spend is imposed.}\\

For mathematical reasons, we do prefer to write the stochastic optimal control problem as a minimization problem:
\begin{equation*}
\inf_{(b_t)_t \in \mathcal{A}}\mathbb E \left[-I_T + K \min\left(S_T,0\right)^2\right],
\end{equation*}
where $\mathcal{A}$ is the set of predictable processes with values in $\mathbb{R}_+ \cup \{+\infty\}$.\\

\emph{Value function:}\\

To tackle this stochastic optimal control model, we introduce the value function
\begin{equation*}
u: (t,I,S) \in [0,T]\times \mathbb{N}\times (-\infty, \bar{S}] \mapsto  \inf_{(b_s)_{s\ge t} \in \mathcal{A}_{t}} \mathbb E\left[-I^{b,t,I}_T + K \min\left(S^{b,t,S}_T,0\right)^2\right],\end{equation*} where
$\mathcal{A}_{t}$ is the set of predictable processes on $[t,T]$ with values in $\mathbb{R}_+ \cup \{+\infty\}$, and
\begin{equation*}
dS^{b,t,S}_s = - p_{N_s} \mathbf{1}_{\{b_s> p_{N_s}\}}dN_s, \quad S^{b,t,S}_t = S,
\end{equation*}
\begin{equation*}
dI^{b,t,I}_s = \mathbf{1}_{\{b_s> p_{N_s}\}}dN_s, \quad I^{b,t,I}_t = I.
\end{equation*}

\subsection{Solution of the stochastic optimal control problem}

\subsubsection{A characterization with a non-standard HJB equation}

To solve the above stochastic optimal control problem, we introduce the associated Hamilton-Jacobi-Bellman (HJB) equation which corresponds to the continuous-time equivalent of the Bellman equation (A.2) derived in discrete time from the dynamic programming principle -- see the appendix. Here the HJB equation is:\footnote{For not making this paper a too technical one, we focus on applications and not on the mathematical analysis of the integro-differential Bellman equations we derive. The interested reader can derive existence and uniqueness results by using for instance the advanced viscosity techniques presented in \cite{bi}. It is noteworthy that we are going to approximate the problem by its fluid limit for which the classical approach with weak semi-concave solutions applies -- see \cite{can,evans}.}

\begin{equation}
\label{m1:HJB}
-\partial_t u(t,I,S) - \lambda \inf_{b \in \mathbb R_+}\int_0^b f(p) (u(t,I+1,S-p) - u(t,I,S)) dp = 0,
\end{equation}
with terminal condition $u(T,I,S) = - I + K \min\left(S,0\right)^2$.\\

Eq. (\ref{m1:HJB}) is a non-standard integro-differential HJB equation -- which is similar to Eq. (A.2). Because the objective function is affine in the variable $I$, the dimensionality of the problem can be reduced by considering the ansatz $u(t,I,S) = - I + v(t,S)$.\\

With this ansatz, Eq. (\ref{m1:HJB}) becomes indeed:
\begin{equation}
\label{m1:HJB2}
-\partial_t v(t,S) - \lambda \inf_{b \in \mathbb R_+} \int_0^b f(p) (v(t,S-p) - v(t,S)-1) dp = 0,
\end{equation}
with terminal condition $v(T,S) = K \min\left(S,0\right)^2$.\footnote{Equivalently, Eq. (\ref{m1:HJB2}) can be written $$-\partial_t v(t,S) - \lambda \inf_{b \in \mathbb R_+} \left(-F(b) + \int_0^b f(p) (v(t,S-p) - v(t,S)) dp\right) = 0.$$ }\\

Eq. (\ref{m1:HJB2}) is another non-standard integro-differential Hamilton-Jacobi-Bellman, but with only one spatial dimension instead of two. It is straightforward to verify that it corresponds to the following stochastic optimal control problem:
\begin{equation}
\label{m1:obj2}
\inf_{(b_t)_t \in \mathcal{A}}\mathbb E \left[-\lambda \int_0^T F(b_t) dt + K \min\left(S_T,0\right)^2\right],
\end{equation}
where $\mathcal{A}$ is the set of predictable processes with values in $\mathbb{R}_+ \cup \{+\infty\}$.\\

The change of variables $u(t,I,S) = - I + v(t,S)$ is key to reduce the dimensionality of the problem. For understanding the underlying rationale, let us notice that $-u(t,I,S)$ is (up to the penalization term)  the number of impressions an ad trader should expect to purchase over the time interval $[0,T]$, if, at time $t$, he has already purchased a number of impressions equal to $I$, and has an amount $S$ to be spent optimally over $[t,T]$. Because the number of impressions that will be purchased over $[t,T]$ does not depend on the number of impressions already purchased, except through the budget that has been spent -- or, equivalently, that remains to be spent --, the quantity $-u(t,I,S) - I$ should be independent of $I$. Hence the ansatz $u(t,I,S) = - I + v(t,S)$. In particular, the quantity $-v(t,S)$ represents (up to the penalization term) the number of impressions an ad trader should expect to purchase over the time interval $[t,T]$, if he has, at time $t$, an amount $S$ to be spent optimally over $[t,T]$.\\

\subsubsection{The reasons why we need to go beyond Eq. (\ref{m1:HJB2})}

From the previous paragraphs, one could consider that the problem is almost entirely solved. Because Eq. (\ref{m1:HJB2}) characterizes the optimal bidding strategy, finding the optimal bid in each state $(t,S)$ of the system simply boils down indeed to implementing a backward (monotone) numerical scheme to approximate the solution of Eq. (\ref{m1:HJB2}).\\

However, it is noteworthy that the problem faced by the ad trader is a two-scale one:
\begin{itemize}
\item The macroscopic scale is the tactical scale.\footnote{This scale may instead be regarded as mesoscopic, if we consider that the advertising campaign constitutes the macroscopic scale.} This scale is related to the time horizon~$T$ and to the amount of cash $\bar{S}$ the ad trader has to spend over the time interval $[0,T]$. In practice, for intraday tactics, the order of magnitude is in the range $10^{2}-10^{4}$ s  (seconds) for $T$, and in the range \EUR{$10^{1}-10^{4}$} for $\bar{S}$.
\item The microscopic scale is the scale of each auction, related to the intensity $\lambda$, and to the variables $b$ and $p$ (in the integral term of Eq. (\ref{m1:HJB2})). The order of magnitude for these variables is around $10^3$ s$^{-1}$ for $\lambda$, and in the range \EUR{$10^{-5}-10^{-2}$} for each bid.\footnote{These values may seem quite small at first reading. We remind the reader that the industry standard is to measure bid levels and prices for bulks of one thousand impressions (CPM, or cost-per-mille). For instance, if for a given auction a practitioner talks about a bid of \EUR{$1$}, then in reality that is \EUR{$10^{-3}$}.}
\end{itemize}

If one wants to approximate the solution of Eq. (\ref{m1:HJB2}) on a grid in $(t,S)$, the time step $\Delta t$ and the cash/spatial step $\Delta S$ should be compatible with the microscopic scale. To take account of the fast arrival of auction requests, a natural time step is $\Delta t$ around $10^{-4}$ second. As far as the cash to spend is concerned, if one wants to be precise when computing the optimal bid for each auction, a natural value for $\Delta S$ is around \EUR{$10^{-7}$}, given the previously discussed orders of magnitude for the bids. Given the orders of magnitude considered at the macroscopic scale, we need to consider a numerical scheme on a grid with $10^6-10^8$ points in time and $10^{8}-10^{11}$ points in space, which is computationally extremely costly.\footnote{It is not an issue that the numerical solution of Eq. (\ref{m1:HJB2}) takes far more time (seconds, minutes, ...) to be computed than the maximal time authorized to send a bid to the auction servers (milliseconds). In practice, the solution can be computed in advance, and then be embedded in an in-memory look-up table allowing low-latency requests.} Even though smart numerical methods can be built to avoid computations at each point of the grid, the multi-scale nature of the problem is an important issue in practice.\footnote{This problem is even more critical in the case of the on-line estimation of the parameters -- see \cite{fgllearning}.}  \\

A way to avoid using a computer-intensive numerical method is to look for almost-closed-form approximations. As we will show below, such an approximation can be found by using a first-order Taylor expansion in the integral term of Eq. (\ref{m1:HJB2}).

\subsection{Almost-closed-form approximation of the solution}

In this subsection, we aim at approximating the value function $v$ and the resulting optimal control function $(t,S) \mapsto b^*(t,S)$ corresponding to the optimal bidding strategy -- the latter is characterized  either by \begin{equation}
\label{sol}v(t,S_{t-} - b^*(t,S_{t-})) = v(t,S_{t-}) + 1,\end{equation} if this equation has a solution, or by $b^*(t,S_{t-}) = +\infty$ otherwise, \emph{i.e.} if we are in the case where $\lim_{S \to -\infty} v(t,S) \le v(t,S_{t-}) + 1$.\\

The main idea underlying the approximation we obtain in the following paragraphs is to replace the term $ v(t,S-p) - v(t,S)$ in  Eq. (\ref{m1:HJB2}) by its first-order Taylor expansion $-p \partial_S v(t,S)$. We therefore replace Eq. (\ref{m1:HJB2}) by:

\begin{equation}
\label{m1:HJBc}-\partial_t v(t,S) - \lambda \inf_{b \in \mathbb R_+} - \int_0^b f(p)\left(1+ p \partial_S v(t,S)\right) dp = 0.
\end{equation}
with terminal condition $v(T,S) = K \min\left(S,0\right)^2$.\\

Intuitively, this approximation is relevant because of the multi-scale nature of the problem: the range of values for $b$ should be several orders of magnitude smaller than the values taken by $S$.\footnote{This point will be discussed below.}

\subsubsection{A fluid-limit approximation}

In the following paragraphs, we show that Eq. (\ref{m1:HJBc}) is the Hamilton-Jacobi equation associated with a variational problem which can be regarded as the fluid limit of the stochastic optimal control problem we considered in the previous subsection.\\

For that purpose, let us introduce the following optimization problem:
\begin{equation*}
\inf_{(\tilde{b}_t)_t \in \mathcal{A}_{\text{det}}} -\lambda \int_0^T F(\tilde{b}_t) dt + K \min\left(\tilde{S}^{\tilde{b}}_T,0\right)^2,
\end{equation*}
where $$d\tilde{S}^{\tilde{b}}_t = - \lambda G(\tilde{b}_t)dt, \quad G: x \in \mathbb{R}_+ \cup \{+\infty\} \mapsto \int_0^{x} p f(p) dp,$$ and where $\mathcal{A}_{\text{det}}$ is the set of $\mathcal{F}_0$-measurable processes with values in $\mathbb{R}_+ \cup \{+\infty\}$.\\

The value function $\tilde{v}$ associated with this problem is defined as:
\begin{equation*}
\tilde{v}(t,S) = \inf_{(\tilde{b}_s)_{s\ge t} \in \mathcal{A}_{\text{det}}} -\lambda \int_t^T F(\tilde{b}_s) ds + K \min\left(\tilde{S}^{\tilde{b},t,S}_T,0\right)^2,
\end{equation*}
where $$d\tilde{S}^{\tilde{b},t,S}_s = - \lambda G(\tilde{b}_s)ds, \quad \tilde{S}^{\tilde{b},t,S}_t = S.$$

We have the following Theorem:

\begin{Theorem}
\label{theo}
Let us define
\begin{eqnarray*}
 H(x) &=& \lambda \sup_{b \in \mathbb R_+} \int_0^b f(p)\left(1+ xp\right) dp.\\
  &=& \left\{
                    \begin{array}{ll}
                      - \lambda x \int_0^{-\frac{1}{x}} F(p) dp, & x <0 \\
                      \lambda\left(1+x\int_0^{\infty} pf(p) dp\right), & x \ge 0.
                    \end{array}
                  \right.\\
\end{eqnarray*}
The value function $\tilde{v}$ is given by:
$$\tilde{v}(t,S) = \sup_{x \le 0} \left(Sx - (T-t) H(x) - \frac{x^2}{4K}\right).$$
It is the unique weak semi-concave solution of the Hamilton-Jacobi equation (\ref{m1:HJBc}).\\

Furthermore, the optimal control function $(t,S) \mapsto \tilde{b}^*(t,S)$ is given by the following:
\begin{itemize}
  \item If $S \ge  \lambda (T-t) \int_0^{\infty} pf(p) dp$, then $\tilde{b}^*(t,S) = +\infty$.
  \item If $S <  \lambda (T-t) \int_0^{\infty} pf(p) dp$, then $\tilde{b}^*(t,S) = -\frac{1}{x^*}$, where $x^*$ is characterized by \begin{equation}
\label{m1:characy}
S = (T-t) H'(x^*) + \frac{x^*}{2K}.
\end{equation}
\end{itemize}
\end{Theorem}

\begin{proof}

By using the change of variables $a_s = \lambda G(\tilde{b}_s) \in \mathcal{I} = [0,\lambda \int_0^\infty p f(p) dp]$, we have:

\begin{equation}
\label{lfa}
\tilde{v}(t,S) = \inf_{(a_s)_{s\ge t} \in \mathcal{A}'_{\text{det}}} \int_t^T L\left(a_s\right) ds + K \min\left(\widehat{S}^{a,t,S}_T,0\right)^2,
\end{equation}
where $\mathcal{A}'_{\text{det}}$ is the set of $\mathcal{F}_0$-measurable processes with values in $\mathcal{I}$, where $$d\widehat{S}^{a,t,S}_s = -a_s ds, \quad \widehat{S}^{a,t,S}_t = S,$$ and where the function $L$ is defined by:
  $$L : a \in \mathcal{I}  \mapsto -\lambda F\left(G^{-1}\left(\frac{a}{\lambda}\right)\right).$$

$L$ is continuously differentiable on the interior of $\mathcal{I}$, with:
\begin{eqnarray*}
  L'(a) &=& \left(G^{-1}\right)'\left(\frac{a}{\lambda}\right) F'\left(G^{-1}\left(\frac{a}{\lambda}\right)\right)\\
   &=& -\frac{f\left(G^{-1}\left(\frac{a}{\lambda}\right)\right)}{G'\left(G^{-1}\left(\frac{a}{\lambda}\right)\right)}\\
   &=& -\frac{1}{G^{-1}\left(\frac{a}{\lambda}\right)}.\\
\end{eqnarray*}

In particular, since $G$ is an increasing function, $L'$ is increasing, and, therefore, $L$ is strictly convex.\\

Let us now compute the Legendre-Fenchel transform of $L$:
\begin{eqnarray}
\nonumber L^*(x) &=& \sup_{a \in \mathcal{I} } x a - L(a)\\
\nonumber &=& \sup_{b \in \mathbb{R}_+ } \lambda x G(b) + \lambda F(b)\\
\label{m1:supb}&=& \lambda \sup_{b \in \mathbb{R}_+ } \int_0^b  (1+ xp) f(p) dp\\
\nonumber &=& H(x).
\end{eqnarray}

If $x < 0$, then the first order condition associated with the supremum in Eq. (\ref{m1:supb}) is $f(b^*) \left(1+ b^*x\right)  = 0$. In other words:

$$x < 0 \Rightarrow b^* = - \frac{1}{x}.$$

If $x$ is nonnegative, then $b^* = +\infty$.\\

In the former case, we can write (by using an integration by parts):
\begin{eqnarray*}
  H(x) &=& \lambda \sup_{b \in \mathbb R_+} \int_0^b \left(1+ xp\right) f(p) dp\\
   &=&  \lambda \int_0^{b^*} f(p)\left(1+ xp\right) dp  \\
   &=& -\lambda x \int_0^{-\frac 1x} F(p) dp. \\
\end{eqnarray*}

The Legendre-Fenchel transform of $L$ is therefore $H$, which can indeed be written as:
$$H(x) =\left\{
                    \begin{array}{ll}
                      - \lambda x \int_0^{-\frac{1}{x}} F(p) dp, & x <0 \\
                      \lambda\left(1+x\int_0^{\infty} pf(p) dp\right), & x \ge 0.
                    \end{array}
                  \right.$$

It is straightforward to check that $H$ is a twice continuously differentiable convex function.\footnote{Straightforward computations give:
$$H'(x) =\left\{
                    \begin{array}{ll}
                      \lambda \int_0^{-\frac{1}{x}} pf(p) dp = \lambda G\left(-\frac{1}{x}\right), & x <0 \\
                      \lambda \int_0^{\infty} pf(p) dp, & x \ge 0.\\
                    \end{array}
                  \right.$$

In particular, $H'(0) = \lambda \int_0^{\infty} pf(p) dp$.\\

Because $\lim_{p \to +\infty} p^3 f(p) = 0$, the function $H'$ is continuously differentiable with:
$$H''(x) =\left\{
                    \begin{array}{ll}
                       -\lambda \frac{1}{x^3} f\left(-\frac 1x\right), & x <0 \\
                      0, & x \ge 0.
                    \end{array}
                  \right.$$
 }\\

We know, from classical variational calculus, that $\tilde{v}$ is given by the Hopf-Lax formula:
\begin{eqnarray}
\label{m1:hl}  \tilde{v}(t,S) &=&  \inf_{y \in [S -\lambda \int_0^\infty pf(p) dp(T-t),S]} \left((T-t) L\left(\frac{S-y}{T-t}\right) + K\min(y,0)^2\right).
\end{eqnarray}

By using an infimal convolution, we can transform Eq. (\ref{m1:hl}) into:
\begin{eqnarray}
\nonumber \tilde{v}(t,S) &=& \sup_{x} \left(Sx - (T-t) H(x) - \sup_y \left(yx - K\min(y,0)^2\right)\right)\\
\nonumber &=& \sup_{x} \left(Sx - (T-t) H(x) - \chi_{\mathbb R_-}(x) - \frac{x^2}{4K}\right)\\
\label{m1:vcff} &=& \sup_{x \le 0} \left(Sx - (T-t) H(x) - \frac{x^2}{4K}\right).
\end{eqnarray}

This is the first result of the Theorem.\\

Furthermore, we know that it is the unique weak semi-concave solution of the following Hamilton-Jacobi equation:
\begin{equation*}
-\partial_t \tilde{v}(t,S) + H(\partial_S \tilde{v}(t,S)) = 0,
\end{equation*}
with terminal condition $\tilde{v}(T,S) = K\min(S,0)^2$.\\

This equation is exactly the same as Eq. (\ref{m1:HJBc}), hence the second assertion of the Theorem.\\

Moreover the optimal control in the modified problem (\ref{lfa}) is given by $$a^*(t,S) = H'(\partial_S \tilde{v}(t,S)).$$

Therefore, the optimal control function in the initial problem is:
\begin{equation*}
\tilde{b}^*(t,S) = \left\{
                    \begin{array}{ll}
                       -\frac{1}{\partial_S \tilde{v}(t,S)}, & \partial_S \tilde{v}(t,S) <0 \\
                      +\infty, & \partial_S \tilde{v}(t,S) \ge 0.
                    \end{array}
                  \right.
\end{equation*}

If $S \ge (T-t) H'(0) = \lambda (T-t)  \int_0^{\infty} pf(p) dp$, then the supremum in Eq. (\ref{m1:vcff}) is reached at $x^*=0$. Therefore, $\partial_S \tilde{v}(t,S) = 0$ and $\tilde{b}^*(t,S) = +\infty$.\\

Otherwise, the supremum in Eq. (\ref{m1:vcff}) is reached at $x^*\le 0$ uniquely characterized by:
\begin{equation*}
S = (T-t) H'(x^*) + \frac{x^*}{2K}.
\end{equation*}

In particular, $\partial_S \tilde{v}(t,S) = x^*$ and $\tilde{b}^*(t,S) = -\frac{1}{x^*}$ where $x^*$ is characterized by Eq. (\ref{m1:characy}).\\\qed\\
\end{proof}

\subsubsection{The approximation of the optimal bidding strategy and its interpretation}

Theorem \ref{theo} invites to use $\tilde{v}$ as an approximation for $v$. Then the optimal bidding strategy $(b^*_t)_t$ can be approximated either by $\tilde{b}^*(t,S_{t-})$, or by solving Eq. (\ref{sol}) in which $v$ is replaced by $\tilde{v}$. In what follows, we consider the former approach. Furthermore, we consider the limit case where $K\to +\infty$. Subsequently, we approximate the optimal bid $b^*_t$ at time $t$ by $\tilde{b}^*_\infty(t,S_{t-})$, where the function $\tilde{b}^*_\infty$ is defined by:\\
\begin{itemize}
  \item If $S \ge  \lambda (T-t) \int_0^{\infty} pf(p) dp$, then $\tilde{b}^*_\infty(t,S) = +\infty$.
  \item If $S <  \lambda (T-t) \int_0^{\infty} pf(p) dp$, then $\tilde{b}^*_\infty(t,S)$ is given by: $$\tilde{b}^*_\infty(t,S) = - \frac{1}{{H'}^{-1}\left(\frac{S}{T-t}\right)}.$$
\end{itemize}

In order to understand the above results, it is worth recalling that $\int_0^{\infty} pf(p) dp$ is the average price to beat, and $\lambda (T-t)$ is the expected number of auctions the ad trader will participate in between the time $t$ and the terminal time $T$. Therefore, $\lambda (T-t)  \int_0^{\infty} pf(p) dp$ is the expected amount of cash the ad trader will spend if he bids an infinite price in order to win all the auctions for which he receives a request. It is therefore natural, in our risk-neutral setting, to bid $\tilde{b}^*_\infty(t,S_{t-}) = +\infty$ at time $t$ if the remaining amount to spend $S_{t-}$ is greater than $\lambda (T-t)  \int_0^{\infty} pf(p) dp$.\footnote{Considering very large values of $b$ may seem to go against the fluid-limit approximation. However, this is not really the case because the price eventually paid for each auction won remains really small compared to the budget $S$, except perhaps near the terminal time $T$.}\\

If $S_{t-} <  \lambda (T-t) \int_0^{\infty} pf(p) dp$, then the approximation $\tilde{b}^*_\infty(t,S_{t-})$ of the optimal bid $b^*_t$ verifies:

$$\lambda (T-t) \int_0^{\tilde{b}^*_\infty(t,S_{t-})} pf(p) dp = S_{t-}.$$

To understand this bidding strategy, let us consider the situation from time $t$, in the interesting case, \emph{i.e.} when $S_{t-} <  \lambda (T-t) \int_0^{\infty} pf(p) dp$. If we use the above approximation of the optimal bidding strategy, the dynamics of the remaining cash is (for $0 \le t \le s \le T$):

\begin{equation*}
dS_s = - p_{N_s} \mathbf{1}_{\{\tilde{b}^*_\infty(s,S_{s-})> p_{N_s}\}}dN_s, \qquad S_t = S.
\end{equation*}

Therefore:

$$d\mathbb{E}[S_s] = - \mathbb{E}\left[\lambda\int_0^{\tilde{b}^*_\infty(s,S_{s-})} p f(p) dp\right]ds = - \frac{\mathbb{E}[S_{s}]}{T-s} ds.$$

This gives $\mathbb{E}[S_s] = S \frac{T-s}{T-t}$. In other words, the approximation of the optimal bidding strategy is such that, on average, the remaining cash is spent evenly across $[t,T]$.\footnote{In the case of a time-varying arrival rate of auction requests $\lambda_t$, a similar variational problem, solved in~\cite{f2}, shows that the optimal strategy is to spend proportionally to $\lambda_t$.} \\

One of the main consequences of the previous analysis is that optimality is not only characterized by the optimal bid, but also by the speed at which the budget is spent. The latter is very helpful in terms of implementation. Indeed, defining the solution by an \textit{optimal scheduling curve} (also known as budget-pacing, see \cite{f2}) makes it possible to dynamically control the bid by a feedback-control system tracking a target curve which does not depend on the form of the function $f$. With this point of view, the exact knowledge of the optimal bidding function is no more necessary. However, this analysis is based on the fluid-limit approximation. Subsequently, it is valid as far as the number of auction requests is large enough. For audience strategies focusing on small subsets of the inventory (\emph{e.g.} targeting only a few customers, restricting the strategy to \textit{premium inventory} only, purchasing slots for a very specific banner format, etc.), a numerical approximation of the solution of the initial stochastic control problem is needed.\footnote{In that case, the size of the grid can be reduced, because $\lambda$ is small.}\\

\subsection{Numerical examples}
\label{num}
We present a numerical example for an algorithm spending \EUR{$1$} during a period of 100 seconds. Due to the high-frequency nature of the algorithm, considering a larger time frame and a larger budget does not make an important difference from a mathematical standpoint. Notice that the spending rate ($\sim$ \EUR{$500$}/day) -- even though it may seem small given the sizes of overall advertizing budgets -- is realistic in the light of our problem: we are focusing on an algorithm which is already restricted to a particular channel, audience, and subset of inventory. In practice, ad trading desks dynamically allocate the overall budget throughout the day across a large array of algorithms, such as the one presented here.\footnote{The optimal allocation of the overall budget across several bidding algorithms goes beyond the scope of this work. However, we can see the extensions of Section 3 as a way of addressing this issue. A natural alternative is to use multi-armed bandit algorithms such as those proposed in \cite{pa1,pa2,pa3}.}

\subsubsection{Numerical approximations of $v$ and $b^*$}

For our numerical experiments, we assume that the distribution of the price to beat is an exponential distribution\footnote{A Pareto distribution or an extreme-value distribution are also relevant choices.} with parameter $\mu$, \emph{i.e.} $f(p) = \mu e^{-\mu p}$. We chose $\mu = 2\cdot 10^{3}$~\EUR$^{-1}$, which is consistent with a CPM of the order of \EUR $0.5$.\\

We represent in Figure \ref{fig:valuef} a numerical approximation of the solution $v$ of the HJB equation~\eqref{m1:HJB2}.\footnote{For approximating numerically the solution $v$ of the HJB equation~\eqref{m1:HJB2}, we consider a backward-in-time explicit scheme on a $(t,S)$ grid. The minimizer $b^*$ at each point of the grid is found by using the first order condition (\ref{sol}) (and an affine approximation between successive points). The integral is then approximated by a closed form formula by using an affine approximation of $v$ between successive points of the grid and the exponential form of $f$.} As stated above, the quantity $-v(t,S)$ can be interpreted (up to the penalty term) as the expected number of impressions that will be purchased from time $t$ by an algorithm following the optimal strategy, when the remaining budget at time $t$ is $S$.\\

\begin{figure}[h!]
	\centering
		\includegraphics[width = 11cm]{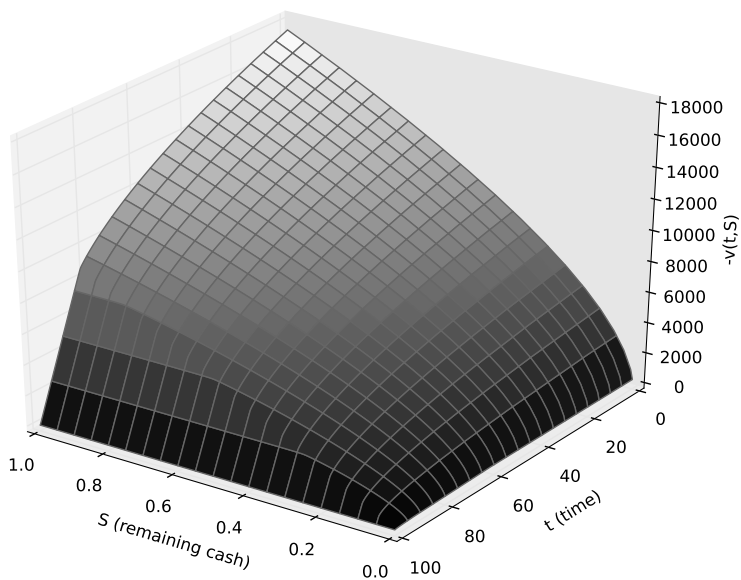}
	\caption{Solution of the HJB equation \eqref{m1:HJB2}. $\lambda= 500 \text{ s}^{-1}$, $T=100 \text{ s}$ and $\bar{S} =$ \EUR{$1$
	}.}
\label{fig:valuef}
\end{figure}

When solving the HJB equation \eqref{m1:HJB2}, the most important quantity to compute is the optimal control, \emph{i.e.} the optimal bid function $(t,S) \mapsto b^*(t,S)$. This function is plotted in Figure \ref{fig:logbid}. Because the values of the optimal bid increase abruptly as $t\rightarrow T$, we prefer to plot the logarithm in basis $10$ of $b^*(t,S)$.\\

\begin{figure}[h!]
	\centering
		\includegraphics[width = 12cm]{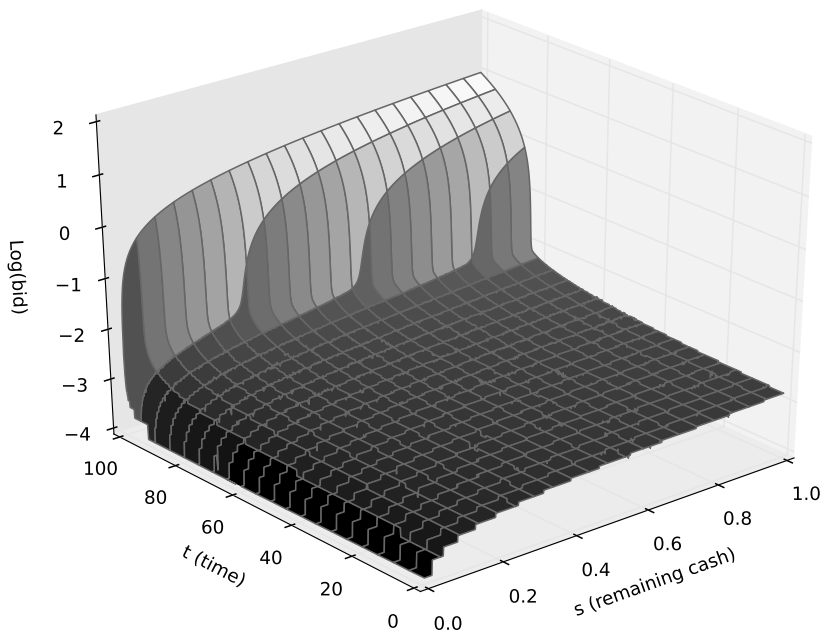}
	\caption{$\log_{10}(b^*(t,S))$ (the unit of $b^*$ is one thousandth of a euro). $\lambda= 500 \text{ s}^{-1}$, $T=100 \text{ s}$ and $\bar{S} =$ \EUR{$1$
	}.}\label{fig:logbid}
\end{figure}

The optimal bids plotted in Figure \ref{fig:logbid} are consistent with intuition: the higher the remaining budget and the closer to the terminal time, the higher the optimal bid.

\subsubsection{Simulations}

One of the consequences of the fluid-limit approximation is that the average spending rate should be constant. An important question is whether or not this feature is also true when one uses the optimal bidding strategy obtained numerically by solving the HJB equation~\eqref{m1:HJB2}. To answer this question, we carried out simulations by drawing the times at which auction requests occur and the values of the price to beat for each auction (using the same set of parameters as above).\\

\begin{figure}[h!]
	\centering
		\includegraphics[width = 7.1cm]{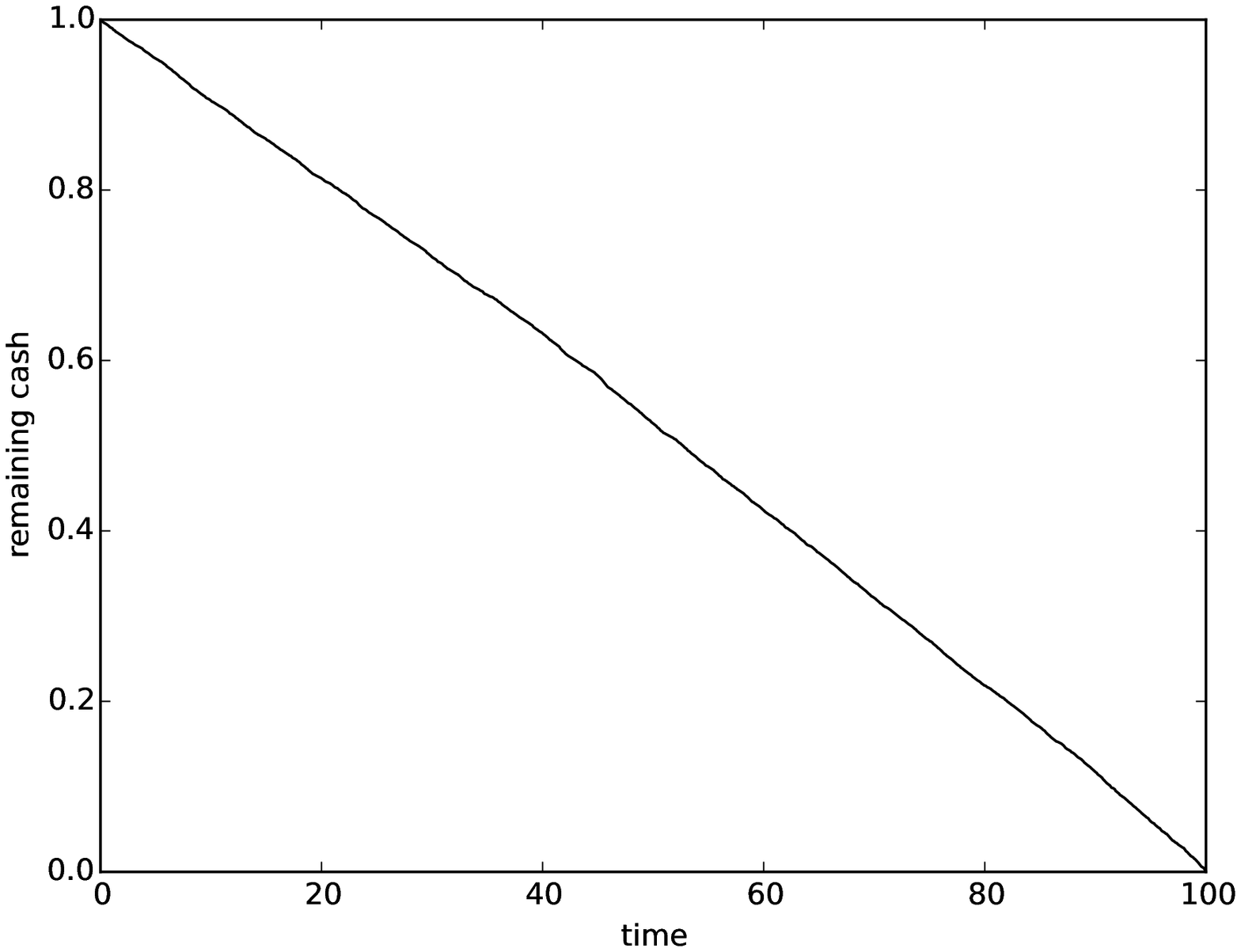}\includegraphics[width = 7.1cm]{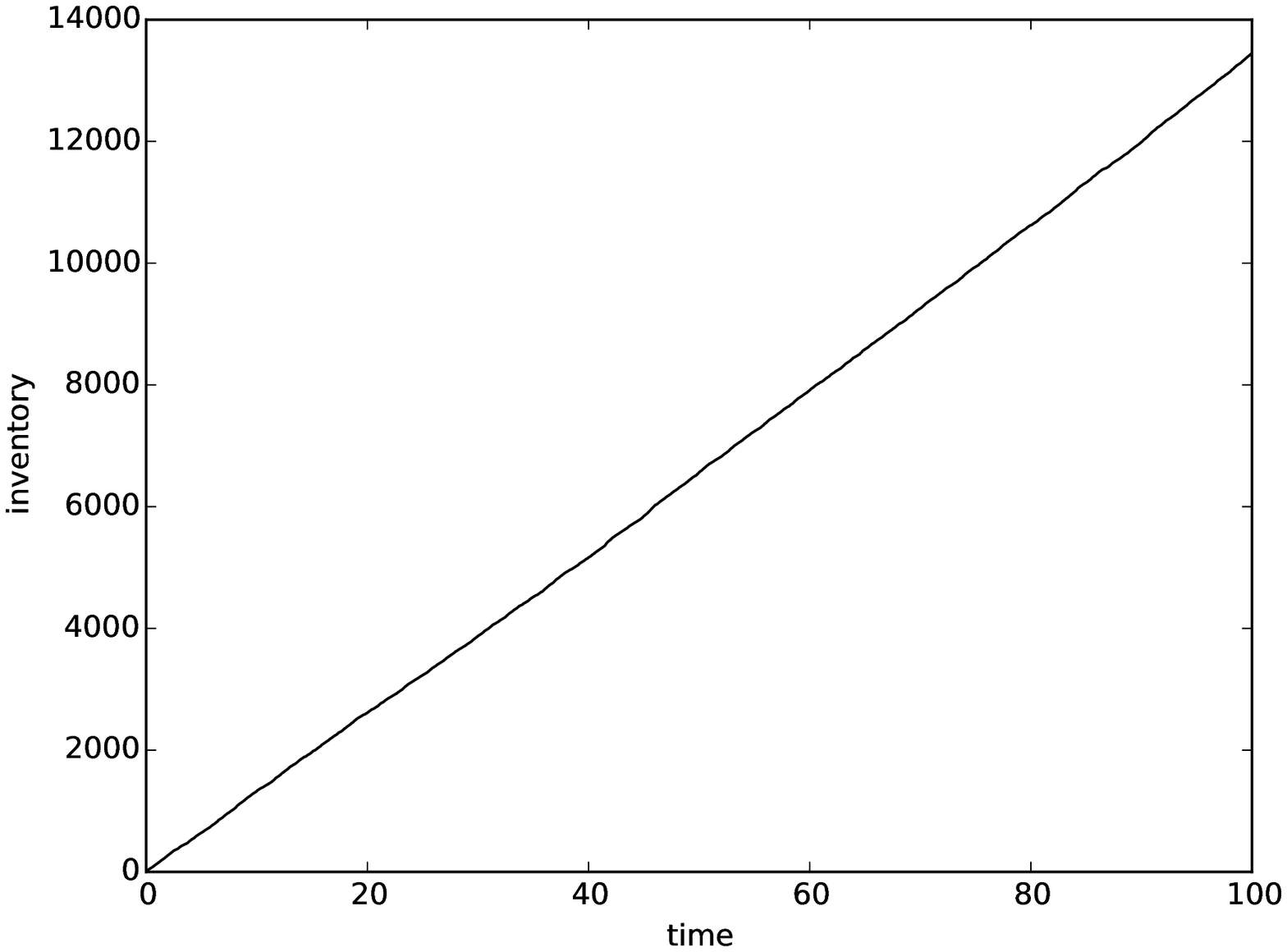}
	\caption{Remaining cash process (left) and inventory process (right) for the optimal bidding strategy computed numerically, over a time window of $100$ seconds.}\label{fig:app1}
\end{figure}

\begin{figure}[h!]
	\centering
		\includegraphics[width = 7.1cm]{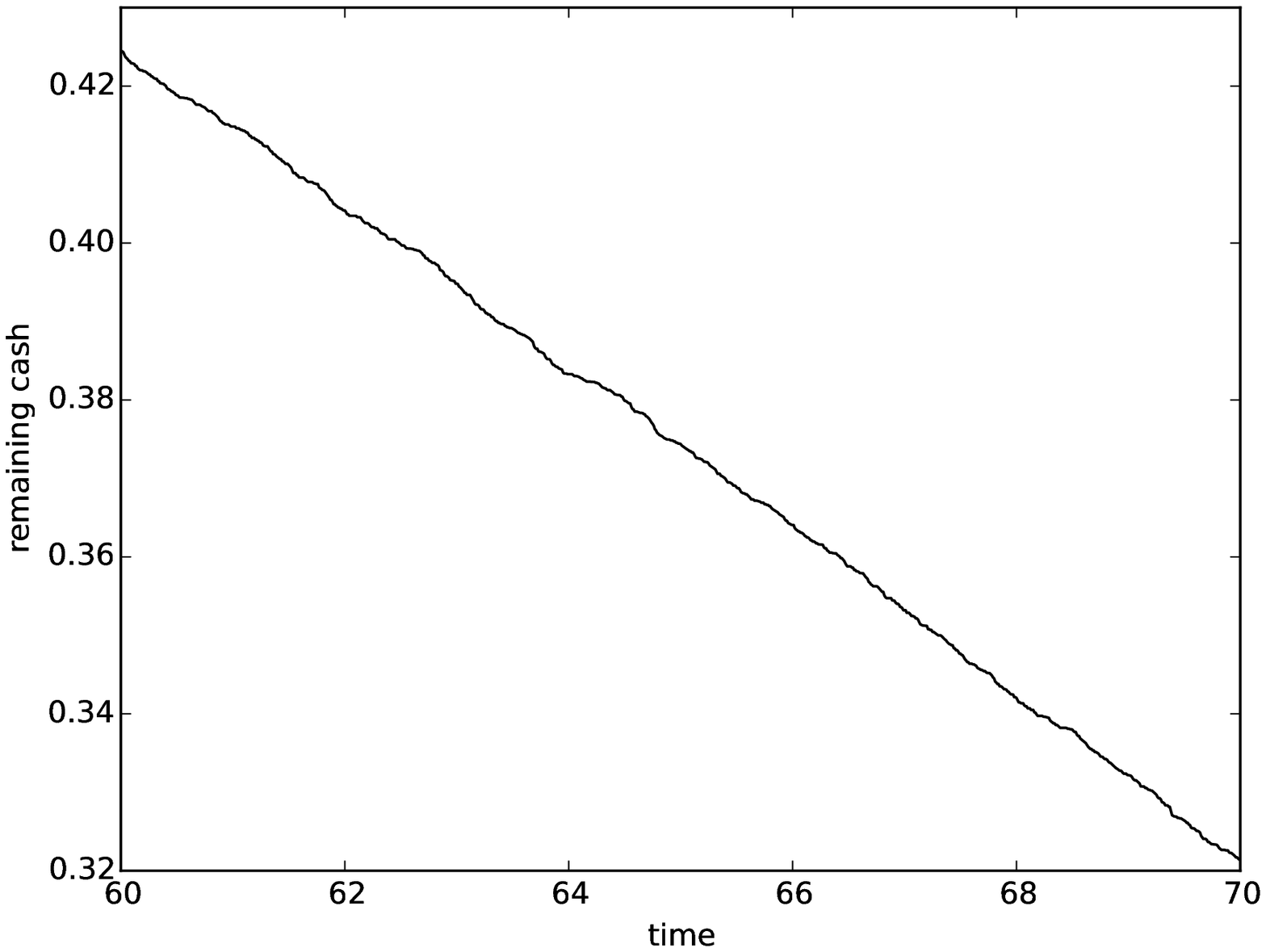}\includegraphics[width = 7.1cm]{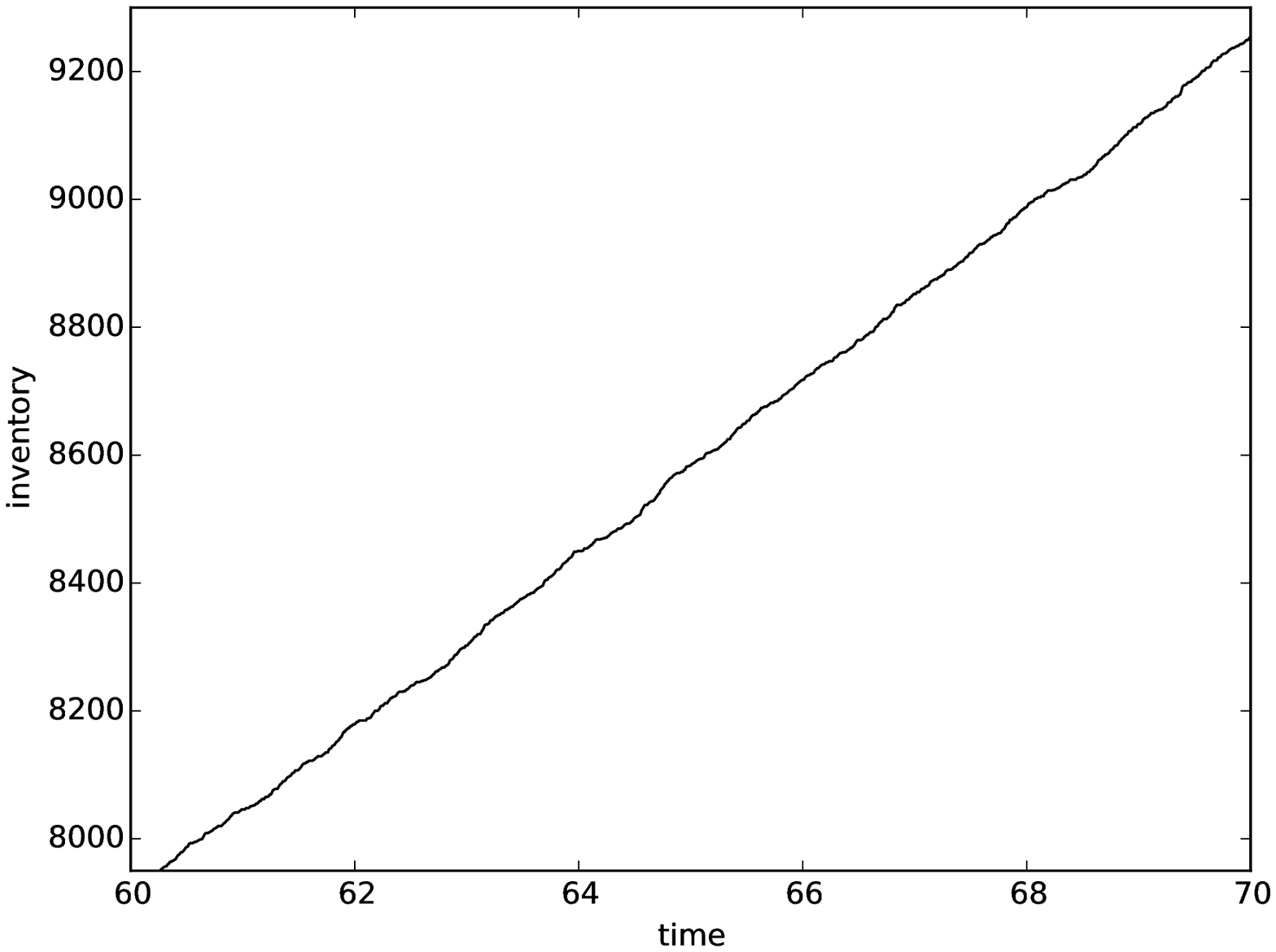}
	\caption{Remaining cash process (left) and inventory process (right) for the optimal bidding strategy computed numerically -- zoom over 10 seconds.}\label{fig:app2}
\end{figure}

\begin{figure}[h!]
	\centering
		\includegraphics[width = 7.1cm]{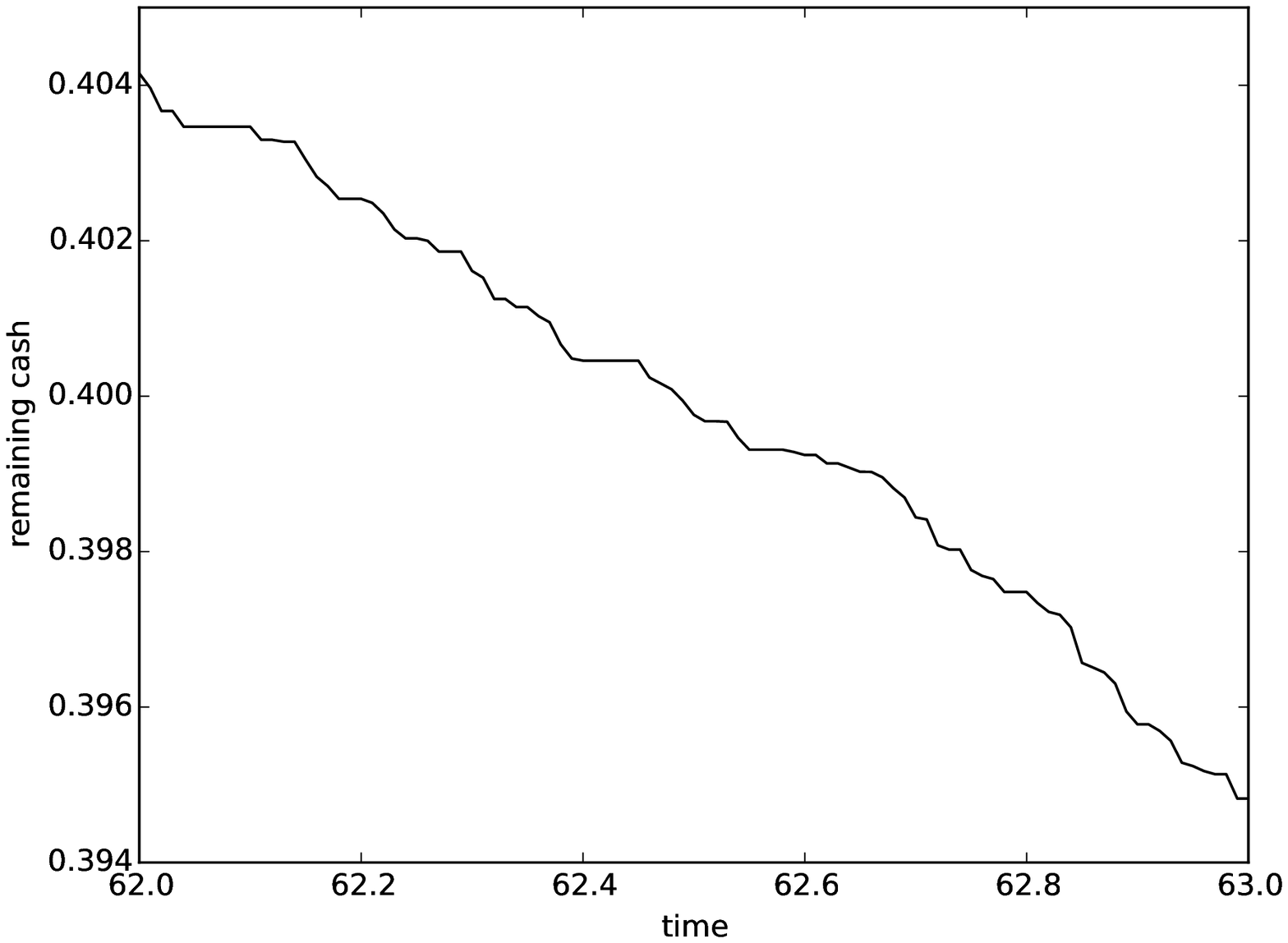}\includegraphics[width = 7.1cm]{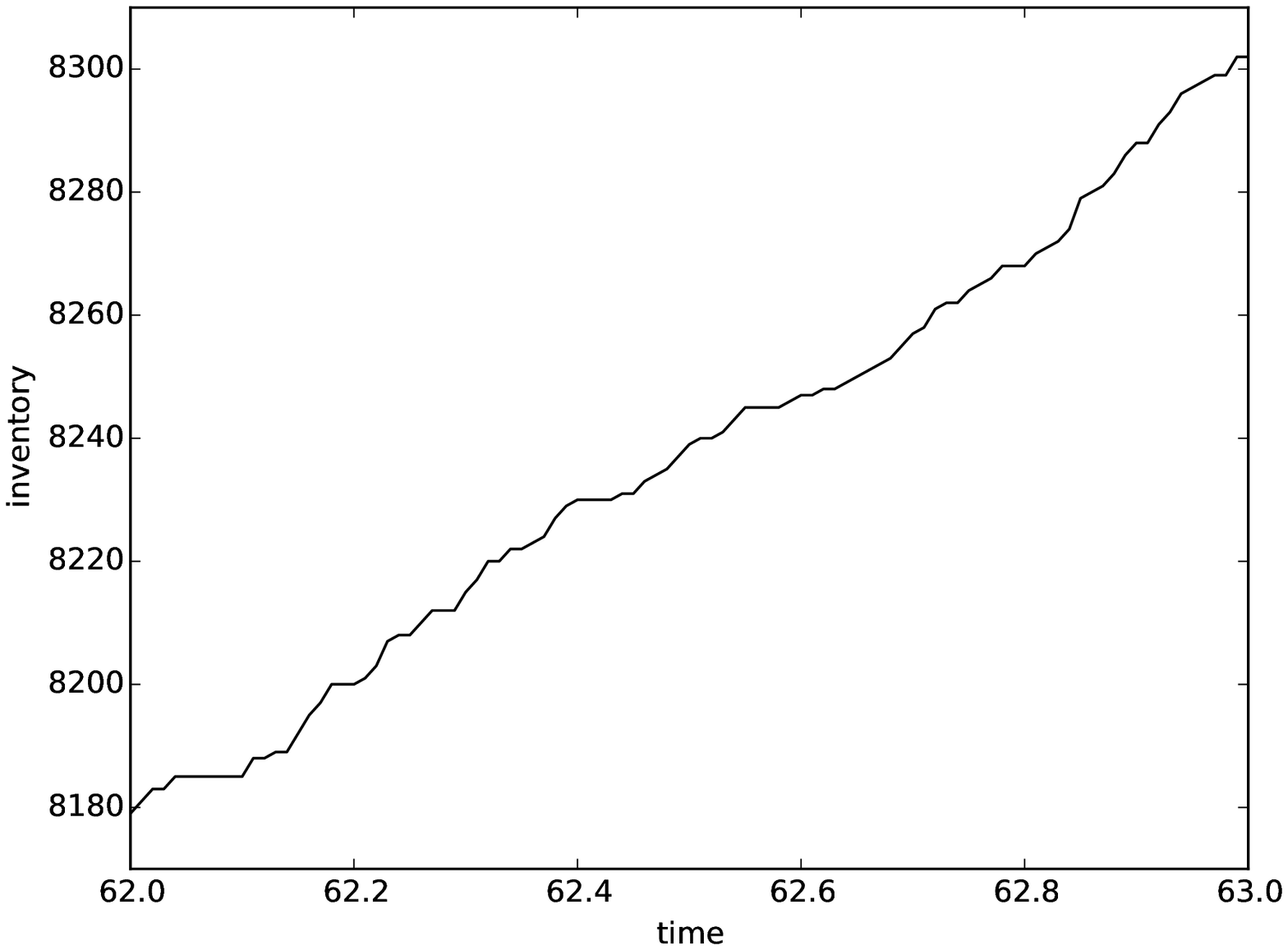}
\caption{Remaining cash process (left) and inventory process (right) for the optimal bidding strategy computed numerically -- zoom over 1 second.}\label{fig:app3}
\end{figure}

Figures \ref{fig:app1},  \ref{fig:app2} and  \ref{fig:app3} show the result for different time windows (we zoom in from 100 s, to 10 s, to 1 s): the remaining-cash process is plotted in the left panels and the number of impressions is plotted in the right panels. For the (realistic) values of the parameters we considered, it is clear that the budget is spent linearly.\\

We also carried out simulations for a smaller value of $\lambda$ ($\lambda = 2 \text{ s}^{-1}$) -- \emph{i.e.} in a less liquid market, with a smaller number of auction requests -- in order to test the robustness of the qualitative result obtained with the fluid-limit approximation. In Figure \ref{fig:low2}, we plot the solution $v$ to the HJB equation \eqref{m1:HJB2} and the associated optimal bids (in log).       In Figure~\ref{fig:low1}, we plot the outcome of the simulations corresponding to the new values of the parameters.\\

As above, we see that spending rate fluctuates but remains around a constant corresponding to even spending over the trading period.\\

\begin{figure}[h!]
	\centering
		\includegraphics[width = 8cm]{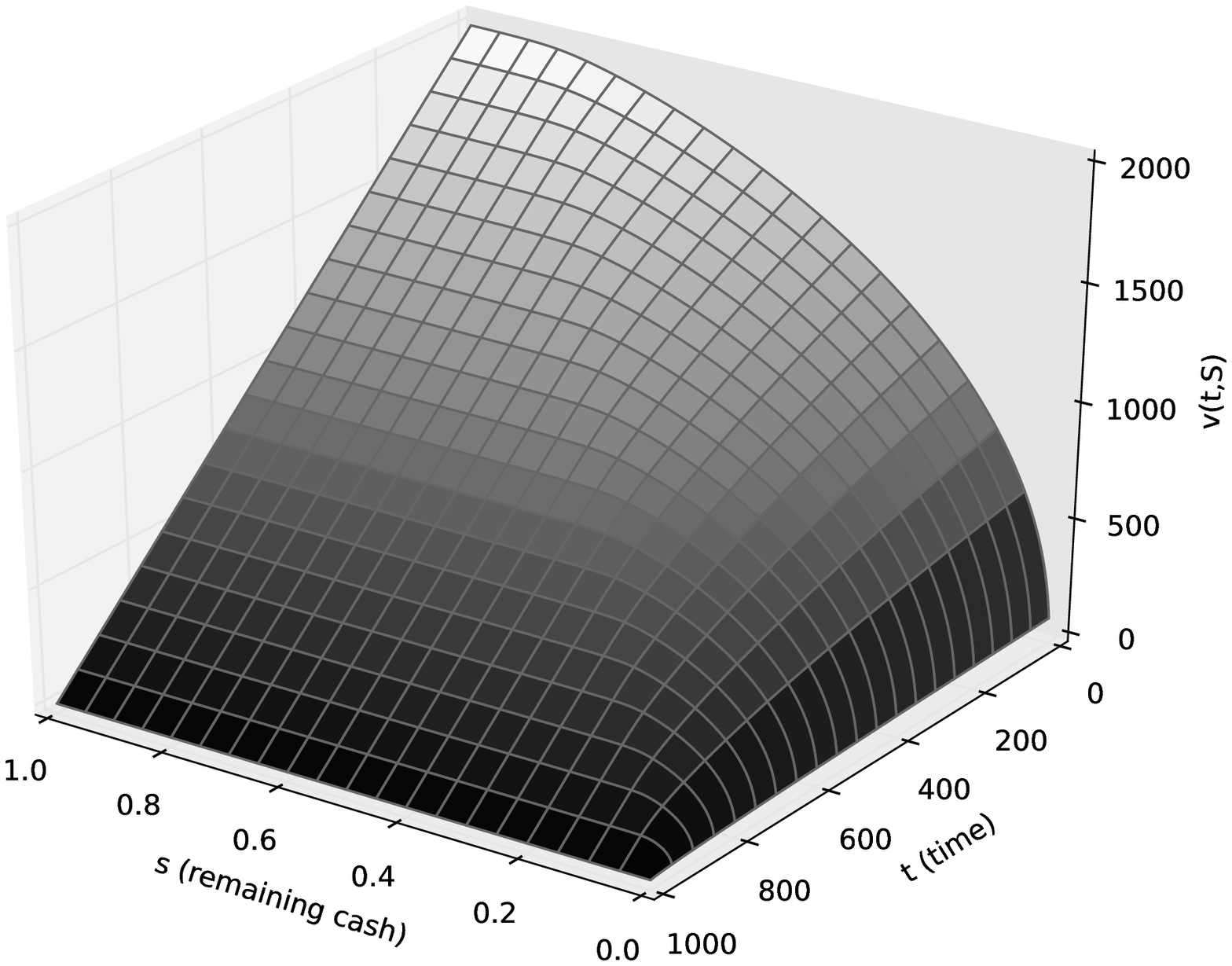}\includegraphics[width = 8cm]{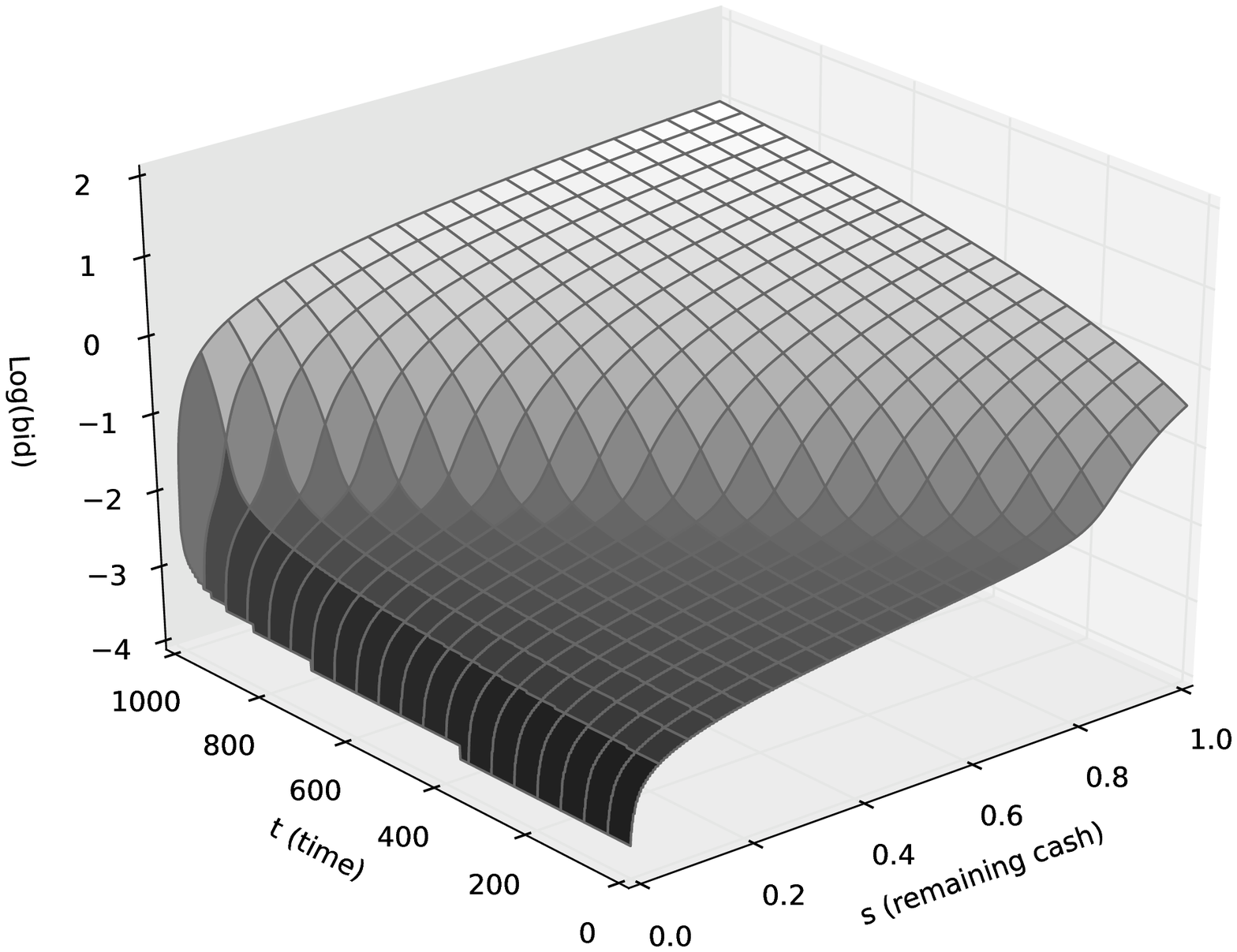}
	\caption{Solution of the HJB equation \eqref{m1:HJB2} (left panel) and optimal bid $\log_{10}(b^*(t,S))$ (right panel). $\lambda= 2 \text{ s}^{-1}$, $T=1000 \text{ s}$ and $\bar{S} =$ \EUR{$1$}. }\label{fig:low2}
\end{figure}
\vspace{0.7cm}
\begin{figure}[h!]
	\centering
		\includegraphics[width = 7.1cm]{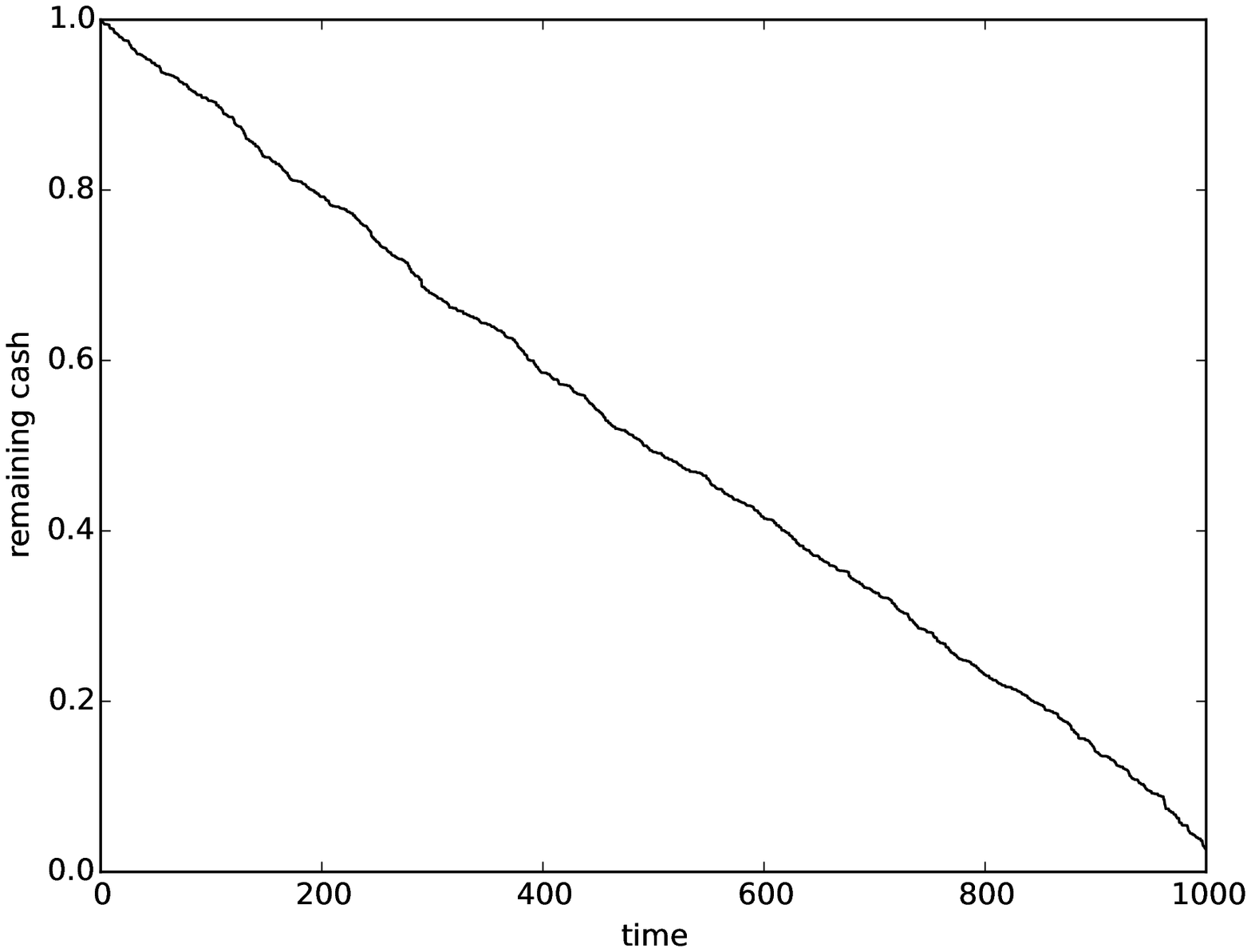}\includegraphics[width = 7.1cm]{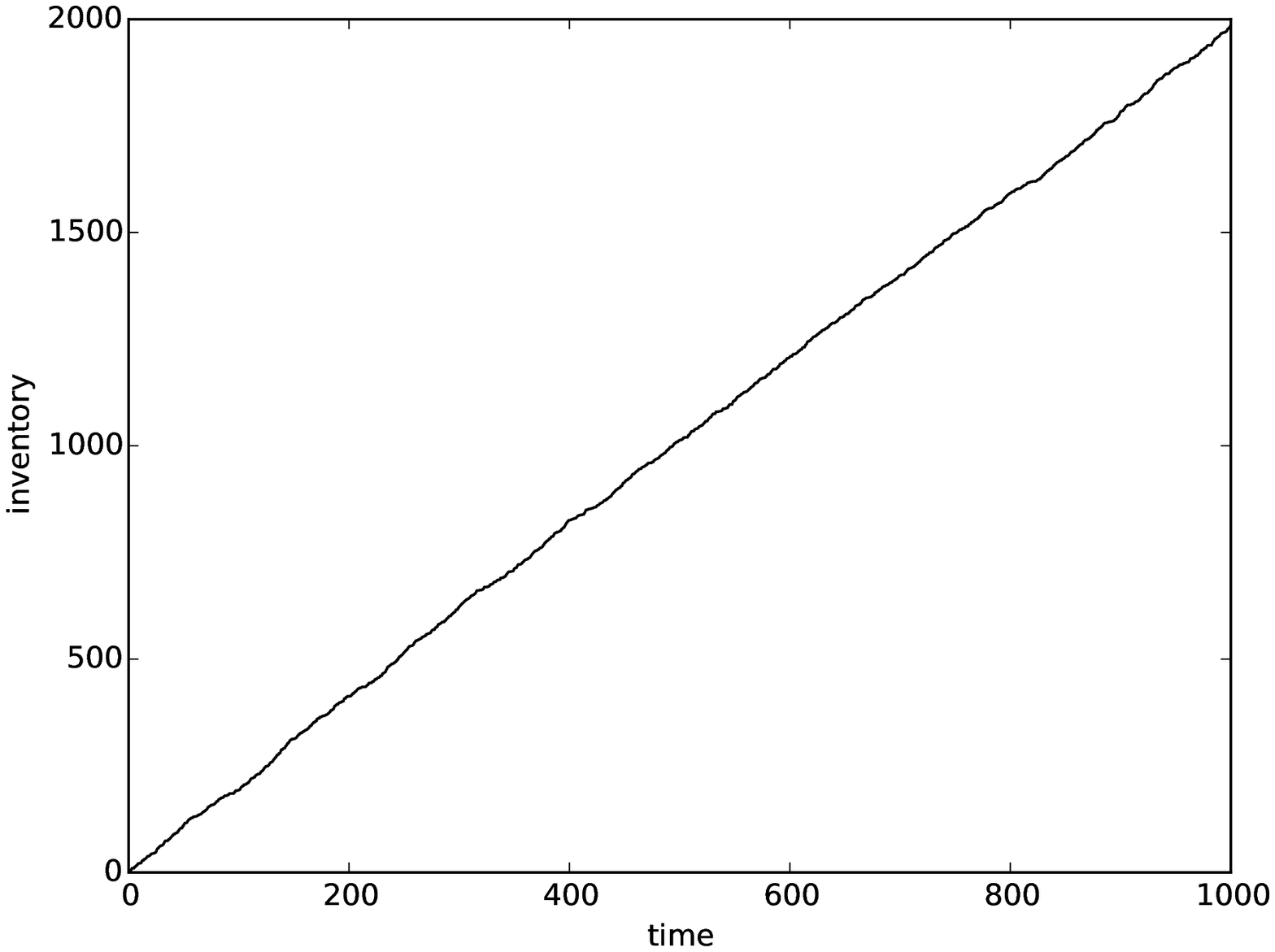}
\caption{Remaining cash process (left) and inventory process (right) for the optimal bidding strategy computed numerically. $\lambda= 2 \text{ s}^{-1}$, $T=1000 \text{ s}$ and $\bar{S} =$ \EUR{$1$}.}\label{fig:low1}
\end{figure}

\section{Extensions and discussion about the model}

In this section, we present two important extensions of the model introduced in Section~2. The first one is a multi-source extension: we consider the case where auction requests arise from several ad exchanges (or from the same ad exchange but from different types of audience), with different frequencies, and  different laws for the price to beat. The second extension is related to conversions: in addition to the number of impressions they buy, ad trading desks often consider the number of conversions as a KPI, which is more in line with the return-on-investment goals of the brand or company launching the advertising campaign. For both extensions, we show that the problem remains in dimension $2$ (after a change of variables), even though the dimension of the state space increases proportionally with the number of sources and almost doubles when conversions are taken into account.\\

In addition to these two extensions, we discuss many features of the model. In particular, the difference between Vickrey auctions and first-price auctions is addressed, along with the existence of floor prices of several kinds. KPIs are finally discussed because the dimensionality reduction at work in the model of Section 2 and in the two extensions presented below is associated with the linearity of the objective function (up to the penalty term).\\

\subsection{Introducing heterogeneity in auctions requests}

The first extension we consider focuses on the situation where the inventory can be purchased from different sources, \emph{i.e.} different ad exchanges, or simply different types of audience. The heterogeneity comes from differences in the frequency of arrival of auction requests, in the distribution of the price to beat associated with each source, or in the relative importance given to each type of inventory in the objective function. An important situation is when the inventory is homogeneous across different ad exchanges and the differences only lie in the statistical properties of the auctions. In that case, we show -- in the fluid-limit case -- that the optimal bid are the same across venues, which is consistent with the result obtained in \cite{f1}. However, when audiences are heterogenous, we show that the optimal bids differ from one source to the next as a function of the relative importance given to the different types of audience in the objective function.\\

\subsubsection{Setup of the model}

In what follows, we still consider an ad trader who wishes to spend a given amount of money~$\bar{S}$ over a time window $[0,T]$.\\

\emph{Auctions:}\\

This ad trader is connected to $J>1$ sources from which he receives requests to participate in auctions in order to purchase inventory -- we assume that the trader knows from which source each auction request arises. Requests are modeled with $J$ marked Poisson processes: the arrival of requests from the source $j \in \lbrace 1, \ldots, J \rbrace$ is triggered by the jumps of the Poisson process $(N^j_t)_t$ with intensity $\lambda^j > 0$, and the marks $(p^j_n)_{n \in \mathbb{N}^*}$ correspond, for each auction request sent by the source $j$, to the highest bid sent by the other participants.\footnote{Throughout this extension and the next one, the superscripts $j \in \lbrace 1, \ldots, J\rbrace$ designate the sources and not exponents.}\\

Every time he receives from the source $j$ a request to participate in an auction, the ad trader can bid a price: at time $t$, if he receives a request from the source $j$, then we denote his bid by $b^j_t$. As in the model with only one source, we assume that the ad trader stands ready to bid (possibly a bid equal to $0$ or $+\infty$) at all times. In particular, we assume that for each $j \in \lbrace 1, \ldots, J \rbrace$, the process $(b^j_t)_t$ is a predictable process with values in $\mathbb{R}_+ \cup \{+\infty\}$.\\

If at time $t$ the $n^{\text{th}}$ auction associated with the source $j$ occurs, the outcome of this auction is the following:
\begin{itemize}
  \item If $b^j_t > p^j_{n}$, then the ad trader wins the auction: he pays the price $p^j_{n}$ and his banner is displayed.
  \item If $b^j_t \le p^j_{n}$, then another trader wins the auction.
\end{itemize}

Like in the model with only one source, we assume that, for each $j \in \lbrace 1, \ldots, J \rbrace$, $(p^j_n)_{n \in \mathbb{N}^*}$ are \emph{i.i.d.} random variables distributed according to an absolutely continuous distribution. We denote by $F^j$ the cumulative distribution function and by $f^j$ the probability density function associated with the source $j$. As above, we assume, for each $j \in \lbrace 1, \ldots, J \rbrace$, that:
\begin{itemize}
  \item $\forall n\in \mathbb N^*$, $p^j_n$ is almost surely positive. In particular, $F^j(0) = 0$.
  %\item $\forall n\in \mathbb N^*$, $p_n \in L^2(\Omega)$.
  \item $\forall p > 0, f^j(p) > 0$.
  \item $\lim_{p \to +\infty} p^3 f^j(p) = 0$.\\
\end{itemize}
We also assume that the random variables $(p^j_n)_{j \in \lbrace 1, \ldots, J \rbrace, n \in \mathbb{N}^*}$ are all independent.\\

\emph{Remaining cash process:}\\

As above, we denote by $(S_t)_t$ the process modeling the amount of cash to be spent. Its dynamics is:
\begin{equation*}
dS_t = - \sum_{j=1}^J p^j_{N^j_t} \mathbf{1}_{\{b^j_t> p^j_{N^j_t}\}}dN^j_t, \quad S_0 = \bar{S}.
\end{equation*}

\emph{Inventory processes:}\\

For each $j \in \lbrace 1, \ldots, J \rbrace$, the number of impressions associated with the auction requests coming for the source $j$ is modeled by an inventory process $(I^j_t)_t$. For each $j \in \lbrace 1, \ldots, J \rbrace$, the dynamics of $(I^j_t)_t$ is:
\begin{equation*}
dI^j_t = \mathbf{1}_{\{b^j_t> p^j_{N^j_t}\}}dN^j_t, \quad I^j_0 = 0.
\end{equation*}

To simplify exposition, we write $I_t = (I^1_t, \ldots, I^J_t) \in \mathbb{N}^J$, and we write the dynamics of the process $(I_t)_t$:

$$dI_t = \sum_{j=1}^J \mathbf{1}_{\{b^j_t> p^j_{N^j_t}\}}dN^j_t e^j,$$ where $(e^1, \ldots, e^J)$ is the canonical basis of $\mathbb{R}^J$.\\

\vspace{2cm}

\emph{Stochastic optimal control problem:}\\

In this first extension of the model, the trader aims at maximizing the expected value of an indicator of the form
$$\alpha^1 I^1_T + \ldots + \alpha^J I^J_T, \quad \alpha^1,\ldots, \alpha^J \ge 0.$$
As in Section 2, we consider a relaxed problem with a penalty for extra-spending:
\begin{equation*}
\inf_{(b^1_t, \ldots, b^J_t)_t \in \mathcal{A}^J}\mathbb E \left[-\sum_{j=1}^J \alpha^j I^j_T + K \min\left(S_T,0\right)^2\right].
\end{equation*}

\subsubsection{HJB equations: from dimension $J+2$ to dimension $2$}

The value function associated with this problem is:
\begin{equation*}
u: (t,I,S) \in [0,T]\times \mathbb{N}^J\times (-\infty, \bar{S}] \mapsto  \inf_{(b^1_s, \ldots, b^J_s)_{s\ge t} \in \mathcal{A}_{t}^J} \mathbb E\left[-\sum_{j=1}^J \alpha^j {I_T^j}^{b,t,I^j} + K \min\left(S^{b,t,S}_T,0\right)^2\right],\end{equation*}
where
\begin{equation*}
dS^{b,t,S}_s = - \sum_{j=1}^J p^j_{N^j_s} \mathbf{1}_{\{b^j_s> p^j_{N^j_s}\}}dN^j_s, \quad S^{b,t,S}_t = S,
\end{equation*}
\begin{equation*}
d{I^j_s}^{b,t,I^j} = \mathbf{1}_{\{b^j_s> p^j_{N^j_s}\}}dN^j_s, \quad {I^j_t}^{b,t,I^j} = I^j, \quad \forall j \in \lbrace 1, \ldots, J \rbrace.
\end{equation*}

The associated Hamilton-Jacobi-Bellman equation is:

\begin{equation}
\label{m1:HJBmulti}
-\partial_t u(t,I,S) - \sum_{j=1}^J \lambda^j \inf_{b^j \in \mathbb R_+}\int_0^{b^j} f^j(p) (u(t,I+e^j,S-p) - u(t,I,S)) dp = 0,
\end{equation}
with terminal condition $u(T,I^1,\ldots,I^J,S) = - \sum_{j=1}^J \alpha^j I^j + K \min\left(S,0\right)^2$.\\

Eq. (\ref{m1:HJBmulti}) is a non-standard integro-differential HJB equation in dimension $J+2$ which generalizes Eq. (\ref{m1:HJB}). As in the case with only one source, we can consider an ansatz of the form $$u(t,I^1,\ldots,I^J,S) = - \sum_{j=1}^J \alpha^j I^j + v(t,S).$$

With this ansatz, Eq. (\ref{m1:HJBmulti}) becomes indeed:
\begin{equation}
\label{m1:HJBmulti2}
-\partial_t v(t,S) - \sum_{j=1}^J \lambda^j \inf_{b^j \in \mathbb R_+} \int_0^{b^j} f^j(p) (v(t,S-p) - v(t,S)-\alpha^j) dp = 0,
\end{equation}
with terminal condition $v(T,S) = K \min\left(S,0\right)^2$.\\

Eq. (\ref{m1:HJBmulti2}) is another HJB equation, but in dimension $2$, which generalize Eq. (\ref{m1:HJB2}).

\subsubsection{Fluid limit approximation}

As in the model with only one source, one could numerically approximate the solution to Eq. (\ref{m1:HJBmulti2}) and subsequently find the optimal bidding strategy. However, it is interesting to consider the approximation arising from  the replacement of  the term $ v(t,S-p) - v(t,S)$ in  Eq. (\ref{m1:HJBmulti2}) by its first-order Taylor expansion $-p \partial_S v(t,S)$. We therefore replace Eq. (\ref{m1:HJBmulti2}) by:

\begin{equation}
\label{m1:HJBmultic}-\partial_t v(t,S) - \sum_{j=1}^J \lambda^j \inf_{b^j \in \mathbb R_+} - \int_0^{b^j} f^j(p) (\alpha^j + p \partial_S v(t,S)) dp = 0,
\end{equation}
with terminal condition $v(T,S) = K \min\left(S,0\right)^2$.\\

Eq. (\ref{m1:HJBmultic}) is related to the fluid limit of the above stochastic optimal control problem which is defined by:

\begin{equation*}
\inf_{(\tilde{b}^1_t,\ldots, \tilde{b}^J_t)_t \in \mathcal{A}_{\text{det}}^J} -\sum_{j=1}^J\lambda^j \alpha^j \int_0^T F^j(\tilde{b}^j_t) dt + K \min\left(\tilde{S}^{\tilde{b}}_T,0\right)^2,
\end{equation*}
where $$d\tilde{S}^{\tilde{b}}_t = - \sum_{j=1}^J\lambda^j G^j(\tilde{b}^j_t)dt, \quad G^j: x \in \mathbb{R}_+ \cup \{+\infty\} \mapsto \int_0^{x} p f^j(p) dp.$$

The value function $\tilde{v}$ associated with this problem is:

\begin{equation*}
\tilde{v}(t,S) = \inf_{(\tilde{b}^1_s, \ldots, \tilde{b}^J_s )_{s\ge t} \in \mathcal{A}_{\text{det}}} -\sum_{j=1}^J\lambda^j \alpha^j \int_t^T F^j(\tilde{b}^j_s) ds + K \min\left(\tilde{S}^{\tilde{b},t,S}_T,0\right)^2,
\end{equation*}
where $$d\tilde{S}^{\tilde{b},t,S}_s = - \sum_{j=1}^J\lambda^j G^j(\tilde{b}^j_s)ds, \quad \tilde{S}^{\tilde{b},t,S}_t = S.$$

Using the same techniques as for Theorem \ref{theo}, we can prove the following Theorem:

\begin{Theorem}
\label{theomulti}
Let us define
\begin{eqnarray*}
\forall j \in \lbrace 1, \ldots, J \rbrace, \quad H^j(x) &=& \lambda^j \sup_{b^j \in \mathbb R_+} \int_0^{b^j} f^j(p)\left(\alpha^j+ xp\right) dp,
\end{eqnarray*}
and $$H = H^1 + \ldots + H^J.$$

The value function $\tilde{v}$ is given by:
$$\tilde{v}(t,S) = \sup_{x \le 0} \left(Sx - (T-t) H(x) - \frac{x^2}{4K}\right).$$
It is the unique weak semi-concave solution of the Hamilton-Jacobi equation (\ref{m1:HJBmultic}).\\

Furthermore, the optimal control function $(t,S) \mapsto \tilde{b}^*(t,S)$ is given by the following:
\begin{itemize}
  \item If $S \ge \sum_{j=1}^J\lambda^j (T-t) \int_0^{\infty} pf^j(p) dp$, then: $$\forall j \in \lbrace 1, \ldots, J \rbrace, \quad \tilde{b}^{j*}(t,S) = +\infty.$$
  \item If $S <  \sum_{j=1}^J\lambda^j (T-t) \int_0^{\infty} pf^j(p) dp$, then: $$\forall j \in \lbrace 1, \ldots, J \rbrace, \quad \tilde{b}^{j*}(t,S) = -\frac{\alpha^j}{x^*},$$ where $x^*$ is characterized by \begin{equation}
\label{m1:characmultiy}
S = (T-t) H'(x^*) + \frac{x^*}{2K}.
\end{equation}
\end{itemize}
\end{Theorem}

\begin{proof}

By using the change of variables $a^j_s = \lambda^j G^j(\tilde{b}^j_s) \in \mathcal{I}^j = [0,\lambda^j \int_0^\infty p f^j(p) dp]$, we have:
\begin{equation}
\label{lfamulti}
\tilde{v}(t,S) = \inf_{(a^1_s, \ldots, a^J_s)_{s\ge t} \in \mathcal{A}'_{\text{det},\mathcal{I}^1, \ldots, \mathcal{I}^J}} \int_t^T \sum_{j=1}^J L^j\left(a^j_s\right) ds + K \min\left(\widehat{S}^{a,t,S}_T,0\right)^2,
\end{equation}
where $\mathcal{A}'_{\text{det},\mathcal{I}^1, \ldots, \mathcal{I}^J}$ is the set of $\mathcal{F}_0$-measurable processes with values in $\mathcal{I}^1 \times \ldots \times \mathcal{I}^J$, where $$d\widehat{S}^{a,t,S}_s = -\left(a^1_s + \ldots + a^J_s\right) ds, \quad \widehat{S}^{a,t,S}_t = S,$$ and where, for all $j \in \lbrace 1, \ldots, J \rbrace$, the function $L^j$ is defined by:
  $$L^j : a^j \in \mathcal{I}^j  \mapsto -\lambda^j \alpha^j F^j\left({G^j}^{-1}\left(\frac{a^j}{\lambda^j}\right)\right).$$

An equivalent way to write Eq. (\ref{lfamulti}) is:
\begin{equation}
\label{lfamulti2}
\tilde{v}(t,S) = \inf_{(\mathfrak{a}_s)_{s\ge t} \in \mathcal{A}''_{\text{det}}} \int_t^T L(\mathfrak{a}_s) ds + K \min\left(\check{S}^{\mathfrak{a},t,S}_T,0\right)^2,
\end{equation}
where $\mathcal{A}''_{\text{det}}$ is the set of $\mathcal{F}_0$-measurable processes with values in $\mathcal{I}=[0,\sum_{j=1}^J\lambda^j \int_0^\infty p f^j(p) dp]$, where $$d\check{S}^{\mathfrak{a},t,S}_s = -\mathfrak{a}_s ds, \quad \check{S}^{\mathfrak{a},t,S}_t = S,$$ and where
$$L: \mathfrak{a} \in \mathcal{I} \mapsto \inf_{(a^1, \ldots, a^J) \in \mathcal{I}^1 \times \ldots \times \mathcal{I}^J, a^1+  \ldots +  a^J = \mathfrak{a} } \sum_{j=1}^J L^j\left(a^j\right).$$

As in the proof of Theorem \ref{theo}, $L^j$ is a convex function, continuously differentiable on the interior of $\mathcal{I}^j$.\\

Let us now compute the Legendre-Fenchel transform of $L^j$:
\begin{eqnarray}
\nonumber {L^j}^*(x) &=& \sup_{a^j \in \mathcal{I}^j } x a^j - L^j(a^j)\\
\nonumber &=& \sup_{b^j \in \mathbb{R}_+ } \lambda^j x G^j(b^j) + \lambda^j \alpha^j F^j(b^j)\\
\label{m1:supbmulti}&=& \lambda^j \sup_{b^j \in \mathbb{R}_+ } \int_0^{b^j}  (\alpha^j+ xp) f^j(p) dp\\
\nonumber &=& H^j(x).
\end{eqnarray}

As in the proof of Theorem \ref{theo}, we can easily find the supremum in Eq. (\ref{m1:supbmulti}):
\begin{itemize}
  \item If $x < 0$, then $b^{j*} = - \frac{\alpha^j}{x}$, and
  \begin{eqnarray*}
  H^j(x) &=& \lambda^j \int_0^{- \frac{\alpha^j}{x}}  (\alpha^j+ xp) f^j(p) dp\\
   &=& -\lambda^j x \int_0^{-\frac{\alpha^j}{x}} F^j(p) dp. \\
\end{eqnarray*}
  \item If $x$ is nonnegative, then $b^{j*} = +\infty$, and
  $$H^j(x) = \lambda^j\left(\alpha^j+ x \int_0^{\infty} p f^j(p) dp\right).$$
\end{itemize}

In particular, $H^j$ is a $C^2$ function.\\

By using an infimal convolution, the function $L$ is convex, and its Legendre transform is $H = H^1 + \ldots + H^J$.\\

From Eq. (\ref{lfamulti2}), $\tilde{v}$ is given by the Hopf-Lax formula:
\begin{eqnarray*}
\tilde{v}(t,S) &=&  \inf_{y \in [S - \sum_{j=1}^J \lambda^j \int_0^\infty pf^j(p) dp (T-t),S]} \left((T-t) L\left(\frac{S-y}{T-t}\right) + K\min(y,0)^2\right).
\end{eqnarray*}

Equivalently, by infimal convolution, we have:
\begin{eqnarray}
\nonumber \tilde{v}(t,S) &=& \sup_{x} \left(Sx - (T-t) H(x) - \sup_y \left(yx - K\min(y,0)^2\right)\right)\\
\label{m1:vcffmulti} &=& \sup_{x \le 0} \left(Sx - (T-t) H(x) - \frac{x^2}{4K}\right).
\end{eqnarray}

This is the first result of the Theorem.\\

Furthermore, we know that $\tilde{v}$ is the unique weak semi-concave solution of the following Hamilton-Jacobi equation:
\begin{equation*}
-\partial_t \tilde{v}(t,S) + H(\partial_S \tilde{v}(t,S)) = 0,
\end{equation*}
with terminal condition $\tilde{v}(T,S) = K\min(S,0)^2$.\\

This equation is exactly the same as Eq. (\ref{m1:HJBmultic}), hence the second assertion of the Theorem.\\

Moreover the optimal control in the modified problem (\ref{lfamulti2}) is given by $\mathfrak{a}^*(t,S) = H'(\partial_S \tilde{v}(t,S))$. Therefore, by classical results on infimal convolutions, the optimal controls in the modified problem (\ref{lfamulti}) are given by:
$$\forall j \in \lbrace 1, \ldots, J \rbrace,\quad a^{j*}(t,S) = {H^j}'(\partial_S \tilde{v}(t,S)).$$

We conclude that the optimal control function in the initial problem is:
\begin{equation*}
\tilde{b}^{j*}(t,S) = \left\{
                    \begin{array}{ll}
                       -\frac{\alpha^j}{\partial_S \tilde{v}(t,S)}, & \partial_S \tilde{v}(t,S) <0 \\
                      +\infty, & \partial_S \tilde{v}(t,S) \ge 0.
                    \end{array}
                  \right.
\end{equation*}

If $S \ge (T-t) H'(0) = \sum_{j=1}^J\lambda^j (T-t)  \int_0^{\infty} pf^j(p) dp$, then the supremum in Eq. (\ref{m1:vcffmulti}) is reached at $x^*=0$. Therefore, $\partial_S \tilde{v}(t,S) = 0$ and $\forall j \in \lbrace 1, \ldots, J \rbrace, \tilde{b}^{j*}(t,S) = +\infty$.\\

Otherwise, the supremum in Eq. (\ref{m1:vcffmulti}) is reached at $x^*\le 0$ uniquely characterized by:
\begin{equation*}
S = (T-t) H'(x^*) + \frac{x^*}{2K}.
\end{equation*}

In particular, $\partial_S \tilde{v}(t,S) = x^*$ and $\forall j \in \lbrace 1, \ldots, J \rbrace, \tilde{b}^{j*}(t,S) = -\frac{\alpha^j}{x^*}$ where $x^*$ is characterized by Eq. (\ref{m1:characmultiy}).\\\qed\\
\end{proof}

Following the same reasoning as in Section 2, we approximate the optimal bids $b^{1*}_t, \ldots, b^{J*}_t$ at time $t$ by using the optimal bids in the fluid-limit approximation and in the limit case $K \to +\infty$. In other words, we approximate the optimal bids $b^{1*}_t, \ldots, b^{J*}_t$ at time $t$ by $\tilde{b}^{1*}_\infty(t,S_{t-}), \ldots, \tilde{b}^{J*}_\infty(t,S_{t-})$, where the functions $(\tilde{b}^{j*}_\infty)_j$ are defined by:\\
\begin{itemize}
  \item If $S \ge  \sum_{j=1}^J\lambda^j (T-t) \int_0^{\infty} pf^j(p) dp$, then: $$\forall j \in \lbrace 1, \ldots, J \rbrace, \quad \tilde{b}^{j*}_\infty(t,S) = +\infty.$$
  \item If $S <  \sum_{j=1}^J\lambda^j (T-t) \int_0^{\infty} pf^j(p) dp$, then: \begin{equation}\label{multichar}\forall j \in \lbrace 1, \ldots, J \rbrace, \quad \tilde{b}^{j*}_\infty(t,S) = - \frac{\alpha^j}{{H'}^{-1}\left(\frac{S}{T-t}\right)}.\end{equation}
\end{itemize}

This approximation of the optimal bidding strategy deserves several remarks.\\

First, the analysis for the case $S_{t-} \ge  \sum_{j=1}^J\lambda^j (T-t) \int_0^{\infty} pf^j(p) dp$ is the same as in the one-source case. The value $\sum_{j=1}^J\lambda^j (T-t) \int_0^{\infty} pf^j(p) dp$ represents the expected amount that will be spent between time $t$ and time $T$ if the ad trader wins all the auctions he participates in. Therefore, it is natural to send a very high bid to be sure to win all the auctions if the amount to spend $S_{t-}$ is greater than that value.\\

Second, if $S_{t-} <  \sum_{j=1}^J\lambda^j (T-t) \int_0^{\infty} pf^j(p) dp$, then we see from Eq. (\ref{multichar}) that the bids differ from one source to the other according to the weights $(\alpha^j)_j$:
$$\forall j,j' \in \lbrace 1, \ldots, J \rbrace,\quad \frac{\tilde{b}^{j*}_\infty(t,S)}{\alpha^j} = \frac{\tilde{b}^{j'*}_\infty(t,S)}{\alpha^{j'}}.$$
In particular, the bids should be the same across sources at a given point in time if one is only interested in the total number of impressions, \emph{i.e.} if $\alpha^1 = \ldots = \alpha^J$.\\

In the case $S_{t-} <  \sum_{j=1}^J\lambda^j (T-t) \int_0^{\infty} pf^j(p) dp$, it is also interesting to notice that Eq. (\ref{multichar}) implies that

$$\forall t \in [0,T], \quad \sum_{j=1}^J\lambda^j (T-t) \int_0^{\tilde{b}^{j*}_\infty(t,S_{t-})} pf^j(p) dp = S_{t-}.$$

If we use the above approximation of the optimal bidding strategy, the dynamics of the remaining cash is (for $0 \le t \le s \le T$):

\begin{equation*}
dS_s = - \sum_{j=1}^J p^j_{N^j_s} \mathbf{1}_{\{b^{j*}_\infty(s,S_{s-})> p^j_{N^j_s}\}}dN^j_s, \qquad S_t = S.
\end{equation*}

Therefore:

$$d\mathbb{E}[S_s] = - \mathbb{E}\left[\sum_{j=1}^J \lambda^j \int_0^{\tilde{b}^{j*}_\infty(s,S_{s-})} p f^j(p) dp ds \right] = - \frac{\mathbb{E}[S_{s}]}{T-s} ds.$$

This gives $\mathbb{E}[S_s] = S \frac{T-s}{T-t}$. As in the one-source case, the approximation of the optimal bidding strategy is such that, on average, the remaining cash is spent evenly across $[t,T]$.

\subsection{What about conversions?}

So far, we have only considered problems where the KPI maximized by the ad trader was linked to the number of impressions he purchases. Practitioners are interested in other KPIs than the number of impressions (related to the CPM). In particular, the number of clicks, or the number of acquisitions of a product (following a click on a banner) are very important indicators of the success of a campaign -- the KPIs used in practice are the CPC and the CPA.\\

In what follows, we consider a model where two variables are optimized upon: the number of impressions, and the number of conversions (which can be regarded as clicks or acquisitions, depending on the considered applications). In fact, we generalize the previous model by considering more general marked Poisson processes, where the marks are not limited to the prices to beat, but also model the occurrence of a conversion with a random variable following a Bernoulli distribution (the parameter of this Bernoulli distribution is known as the \textit{conversion rate}, \emph{i.e.} the probability to turn an impression into a conversion).\\

\subsubsection{Setup of the model}

As above, we still consider an ad trader who wishes to spend a given amount of money $\bar{S}$ over a time window $[0,T]$.\\

\vspace{1cm}

\emph{Auctions:}\\

This ad trader is connected to $J>1$ sources from which he receives requests to participate in auctions in order to purchase inventory -- we assume that the trader knows from which source each auction request arises. Requests are modeled with $J$ marked Poisson processes: the arrival of requests from the source $j \in \lbrace 1, \ldots, J \rbrace$ is triggered by the jumps of the Poisson process $(N^j_t)_t$ with intensity $\lambda^j > 0$, and the marks $(p^j_n)_{n \in \mathbb{N}^*}$ and $(\xi^j_n)_{n \in \mathbb{N}^*}$ correspond, for each auction request sent by the source $j$, respectively to the highest bid sent by the other participants, and to the occurrence of a conversion -- $\xi^j_n \in \{0,1\}$ only makes sense if the auction is won by the ad trader.\\

Every time he receives from the source $j$ a request to participate in an auction, the ad trader can bid a price: at time $t$, if he receives a request from the source $j$, then we denote his bid by $b^j_t$. As in the above model, we assume that for each $j \in \lbrace 1, \ldots, J \rbrace$, the process $(b^j_t)_t$ is a predictable process with values in $\mathbb{R}_+ \cup \{+\infty\}$.\\

If at time $t$ the $n^{\text{th}}$ auction associated with the source $j$ occurs, the outcome of this auction is the following:
\begin{itemize}
  \item If $b^j_t > p^j_{n}$, then the ad trader wins the auction: he pays the price $p^j_{n}$ and his banner is displayed. Moreover, a conversion occurs if and only if $\xi^j_n = 1$.
  \item If $b^j_t \le p^j_{n}$, then another trader wins the auction.
\end{itemize}

As above, we assume that for each $j \in \lbrace 1, \ldots, J \rbrace$, $(p^j_n)_{n \in \mathbb{N}^*}$ are \emph{i.i.d.} random variables distributed according to an absolutely continuous distribution. We denote by $F^j$ the cumulative distribution function and by $f^j$ the probability density function associated with the source $j$. As above, we assume, for each $j \in \lbrace 1, \ldots, J \rbrace$, that:
\begin{itemize}
  \item $\forall n\in \mathbb N^*$, $p^j_n$ is almost surely positive. In particular, $F^j(0) = 0$.
  %\item $\forall n\in \mathbb N^*$, $p_n \in L^2(\Omega)$.
  \item $\forall p > 0, f^j(p) > 0$.
  \item $\lim_{p \to +\infty} p^3 f^j(p) = 0$.\\
\end{itemize}
We also assume that the random variables $(p^j_n)_{j \in \lbrace 1, \ldots, J \rbrace, n \in \mathbb{N}^*}$ are all independent.\\

As far as the variables $(\xi^j_n)_{j \in \lbrace 1, \ldots, J \rbrace, n \in \mathbb{N}^*}$ are concerned, we assume that they are all independent and independent from the variables $(p^j_n)_{j \in \lbrace 1, \ldots, J \rbrace, n \in \mathbb{N}^*}$. Moreover, we assume that for each $j \in \lbrace 1, \ldots, J \rbrace$, $(\xi^j_n)_{n \in \mathbb{N}^*}$ are \emph{i.i.d.} random variables distributed according to a Bernoulli distribution with parameter $\nu^j \in [0,1]$.\\

\emph{Remaining cash process:}\\

As above, we denote by $(S_t)_t$ the process modeling the amount of cash to be spent. Its dynamics is:

\begin{equation*}
dS_t = - \sum_{j=1}^J p^j_{N^j_t} \mathbf{1}_{\{b^j_t> p^j_{N^j_t}\}}dN^j_t, \quad S_0 = \bar{S}.
\end{equation*}

\emph{Inventory processes:}\\

For each $j \in \lbrace 1, \ldots, J \rbrace$, the number of impressions associated with the auction requests coming for the source $j$ is modeled by an inventory process $(I^j_t)_t$. For each $j \in \lbrace 1, \ldots, J \rbrace$, the dynamics of $(I^j_t)_t$ is:
\begin{equation*}
dI^j_t =\mathbf{1}_{\{b^j_t> p^j_{N^j_t}\}}dN^j_t, \quad I^j_0 = 0.
\end{equation*}

\emph{Processes for the number of conversions:}\\

For each $j \in \lbrace 1, \ldots, J \rbrace$, the number of conversions associated with the auction requests coming for the source $j$ is modeled by a new process $(C^j_t)_t$. For each $j \in \lbrace 1, \ldots, J \rbrace$, the dynamics of $(C^j_t)_t$ is:
\begin{equation*}
dC^j_t = \xi^j_{N^j_t} \mathbf{1}_{\{b^j_t> p^j_{N^j_t}\}}dN^j_t, \quad C^j_0 = 0.
\end{equation*}

\emph{Stochastic optimal control problem:}\\

In this second extension of the model, the trader aims at maximizing the expected value of an indicator of the form
$$\left(\alpha^1 I^1_T + \ldots + \alpha^J I^J_T\right) + \left(\delta^1 C^1_T + \ldots + \delta^J C^J_T\right) , \quad \alpha^1,\ldots, \alpha^J \ge 0,\quad \delta^1,\ldots, \delta^J \ge 0.$$

The relaxed problem we consider is:
\begin{equation*}
\inf_{(b^1_t, \ldots, b^J_t)_t \in \mathcal{A}^J}\mathbb E \left[-\sum_{j=1}^J \alpha^j I^j_T - \sum_{j=1}^J \delta^j C^j_T  + K \min\left(S_T,0\right)^2\right].
\end{equation*}

\subsubsection{HJB equations: from dimension $2J+2$ to dimension $2$}

The value function associated with this problem is:
$$
u: (t,I,C,S) \in [0,T]\times \mathbb{N}^J\times \mathbb{N}^J\times (-\infty, \bar{S}]$$$$ \mapsto  \inf_{(b^1_s, \ldots, b^J_s)_{s\ge t} \in \mathcal{A}_{t}^J} \mathbb E\left[-\sum_{j=1}^J \alpha^j {I_T^j}^{b,t,I^j} - \sum_{j=1}^J \delta^j {C^j_T}^{b,t,C^j} + K \min\left(S^{b,t,S}_T,0\right)^2\right],$$
where
\begin{equation*}
dS^{b,t,S}_s = - \sum_{j=1}^J p^j_{N^j_s} \mathbf{1}_{\{b^j_s> p^j_{N^j_s}\}}dN^j_s, \quad S^{b,t,S}_t = S,
\end{equation*}
\begin{equation*}
d{I^j_s}^{b,t,I^j} = \mathbf{1}_{\{b^j_s> p^j_{N^j_s}\}}dN^j_s, \quad {I^j_t}^{b,t,I^j} = I^j, \quad \forall j \in \lbrace 1, \ldots, J \rbrace,
\end{equation*}
and
\begin{equation*}
d{C^j_s}^{b,t,C^j} = \xi^j_{N^j_s} \mathbf{1}_{\{b^j_s> p^j_{N^j_s}\}}dN^j_s, \quad {C^j_t}^{b,t,C^j} = C^j, \quad \forall j \in \lbrace 1, \ldots, J \rbrace.
\end{equation*}

In particular, there are $J$ new state variables corresponding to the number of conversions associated with each of the $J$ sources.\\

The associated Hamilton-Jacobi-Bellman equation is:

\begin{equation*}
-\partial_t u(t,I,C,S) - \sum_{j=1}^J \lambda^j \inf_{b^j \in \mathbb R_+}\int_0^{b^j} f^j(p) \left[(1-\nu^j)(u(t,I+e^j,C,S-p) - u(t,I,C,S))\right.
\end{equation*}
 \begin{equation}
\label{m1:HJBconv}
 \left.+ \nu^j (u(t,I+e^j,C+e^j,S-p) - u(t,I,C,S)) \right] dp = 0,
\end{equation}
with terminal condition $$u(T,I^1,\ldots,I^J, C^1,\ldots,C^J,S) = - \sum_{j=1}^J \alpha^j I^j - \sum_{j=1}^J \delta^j C^j + K \min\left(S,0\right)^2.$$

Eq. (\ref{m1:HJBconv}) is a non-standard integro-differential HJB equation in dimension $2J+2$ which generalizes Eq. (\ref{m1:HJBmulti}). In this extension, we consider an ansatz of the form $$u(t,I^1,\ldots,I^J, C^1,\ldots,C^J,S) = - \sum_{j=1}^J \alpha^j I^j - \sum_{j=1}^J \delta^j C^j + v(t,S).$$

With this ansatz, Eq. (\ref{m1:HJBconv}) becomes another HJB equation, but in dimension $2$, which generalizes Eq. (\ref{m1:HJBmulti2}) to the case where conversions are taken into account:
\begin{equation}
\label{m1:HJBconv2}
-\partial_t v(t,S) - \sum_{j=1}^J \lambda^j \inf_{b^j \in \mathbb R_+} \int_0^{b^j} f^j(p) (v(t,S-p) - v(t,S)-\alpha^j - \nu^j \delta^j) dp = 0,
\end{equation}
with terminal condition $v(T,S) = K \min\left(S,0\right)^2$.

\subsubsection{Fluid limit approximation}

Eq. (\ref{m1:HJBconv2}) is the same as Eq. (\ref{m1:HJBmulti2}), except that $\alpha^j$ is replaced by $\alpha^j + \nu^j \delta^j$. In particular we can use similar rules as in the previous extension to approximate the optimal bidding strategy.\\

We introduce the function $H$ defined by:

\begin{eqnarray*}
H: x \in \mathbb{R} \mapsto  H(x) &=& \sum_{j=1}^J \lambda^j \sup_{b^j \in \mathbb R_+} \int_0^{b^j} f^j(p)\left(\alpha^j + \nu^j \delta^j + xp\right) dp\\
&=& \left\{
                    \begin{array}{ll}
                      \sum_{j=1}^J -\lambda^j x \int_0^{-\frac{\alpha^j + \nu^j \delta^j}{x}} F^j(p) dp , & x <0 \\
                      \sum_{j=1}^J \lambda^j \left(\alpha^j + \nu^j \delta^j+x\int_0^{\infty} pf(p) dp\right), & x \ge 0.
                    \end{array}
                  \right.
\end{eqnarray*}

In the case of conversions, we approximate the optimal bids $b^{1*}_t, \ldots, b^{J*}_t$ at time $t$ by $\tilde{b}^{1*}_\infty(t,S_{t-}), \ldots, \tilde{b}^{J*}_\infty(t,S_{t-})$, where the functions $(\tilde{b}^{j*}_\infty)_j$ are defined by:\\
\begin{itemize}
  \item If $S \ge  \sum_{j=1}^J\lambda^j (T-t) \int_0^{\infty} pf^j(p) dp$, then: $$\forall j \in \lbrace 1, \ldots, J \rbrace, \quad \tilde{b}^{j*}_\infty(t,S) = +\infty.$$
  \item If $S <  \sum_{j=1}^J\lambda^j (T-t) \int_0^{\infty} pf^j(p) dp$, then: \begin{equation*}\forall j \in \lbrace 1, \ldots, J \rbrace, \quad \tilde{b}^{j*}_\infty(t,S) = - \frac{\alpha^j + \nu^j \delta^j}{{H'}^{-1}\left(\frac{S}{T-t}\right)}.\end{equation*}
\end{itemize}

As in the previous extension of the model, if $S_{t-} \ge  \sum_{j=1}^J\lambda^j (T-t) \int_0^{\infty} pf^j(p) dp$, then it is natural to send a very high bid to be sure to win all the auctions.\\

If $S_{t-} <  \sum_{j=1}^J\lambda^j (T-t) \int_0^{\infty} pf^j(p) dp$, then the bids differ from one source to the other according to the following rule:
$$\forall j,j' \in \lbrace 1, \ldots, J \rbrace,\quad \frac{\tilde{b}^{j*}_\infty(t,S)}{\alpha^j+\nu^j \delta^j} = \frac{\tilde{b}^{j'*}_\infty(t,S)}{\alpha^{j'}+\nu^{j'} \delta^{j'}}.$$
In particular, if the ad trader only cares about the total number of conversions, \emph{i.e.} if $\alpha^1 = \ldots = \alpha^J = 0$ and $\delta^1 = \ldots = \delta^J = 1$, then the bids are the same across sources, up to a multiplicative factor which corresponds to the probability of conversion associated with each source.\\

Finally, in this extension as in the previous one, if one uses the fluid-limit approximation, then, the budget is expected to be spent evenly.

\subsection{Discussion about the models}

We have presented a first model in Section 2 and two extensions of that first model earlier in this section. In the following paragraphs, we aim at challenging the assumptions underlying these models. In particular, we discuss the specificities associated with second-price auctions, and the existence of floor prices, which goes against the assumptions on the distribution of the price to beat in the models presented above. We also discuss nonlinear KPIs for which there is no dimensionality reduction.

\subsubsection{First-price auctions vs. second-price auctions}

In both the initial model of Section 2 and the two extensions presented in this section, we considered Vickrey auctions. In other words, the price paid by the ad trader when he wins the auction is not the price he has bid (the first price), but a lower price corresponding to the highest bid that has been proposed by other participants (the second price). In the case of first-price auctions, the dynamics of the remaining budget (in the single-source model of Section 2) is not anymore given by Eq. (\ref{dynS}), but instead by:
\begin{equation*}
dS_t = - b_t \mathbf{1}_{\{b_t> p_{N_t}\}}dN_t, \quad S_0 = \bar{S}.
\end{equation*}

If we consider the same objective function as in Section 2, then the new value function
\begin{equation*}
u: (t,I,S) \in [0,T]\times \mathbb{N}\times (-\infty, \bar{S}] \mapsto  \inf_{(b_s)_{s\ge t} \in \mathcal{A}_{t}} \mathbb E\left[-I^{b,t,I}_T + K \min\left(S^{b,t,S}_T,0\right)^2\right],\end{equation*} where
\begin{equation*}
dS^{b,t,S}_s = - b_s \mathbf{1}_{\{b_s> p_{N_s}\}}dN_s, \quad S^{b,t,S}_t = S,
\end{equation*}
and
\begin{equation*}
dI^{b,t,I}_s = \mathbf{1}_{\{b_s> p_{N_s}\}}dN_s, \quad I^{b,t,I}_t = I,
\end{equation*}
is associated with the following HJB equation:
\begin{equation*}
-\partial_t u(t,I,S) - \lambda \inf_{b \in \mathbb R_+}F(b) (u(t,I+1,S-b) - u(t,I,S)) = 0,
\end{equation*}
with terminal condition $u(T,I,S) = - I + K \min\left(S,0\right)^2$.\\

With the ansatz $u(t,I,S) = - I + v(t,S)$, this equation becomes:
\begin{equation}
\label{m1:HJB2first}
-\partial_t v(t,S) - \lambda \inf_{b \in \mathbb R_+} F(b) (v(t,S-b) - v(t,S)-1) = 0,
\end{equation}
with terminal condition $v(T,S) = K \min\left(S,0\right)^2$.\\

Eq. (\ref{m1:HJB2first}) replaces Eq. (\ref{m1:HJB2}). In particular, the first order condition is not anymore
$$v(t,S-b) - v(t,S)-1 = 0,$$ but instead
$$v(t,S-b) - v(t,S)-1 = \frac{F(b)}{f(b)} \partial_S v(t,S-b).$$

In particular, because $v$ is nonincreasing with respect to $S$, the optimal bid should always be lower in the case of first-price auctions than in the case of second-price auctions.\\

It is also noteworthy that a fluid-limit approximation can be considered in the case of Eq.~(\ref{m1:HJB2first}) by solving

\begin{equation}
\label{m1:HJB2firsttilde}
-\partial_t \tilde{v}(t,S) + \lambda \sup_{b \in \mathbb R_+} (F(b) b \partial_S \tilde{v}(t,S) + F(b)) = 0,
\end{equation}
with terminal condition $\tilde{v}(T,S) = K \min\left(S,0\right)^2$.\\

In particular, Eq. (\ref{m1:HJB2firsttilde}) is a first-order Hamilton-Jacobi equation, and by using the same techniques as in the case of second-price auctions, one also finds that the optimal strategy consists of spending the remaining budget evenly.\\

\subsubsection{Floor prices}

In order to improve performances, publishers can set floor prices which sometimes modify the underlying nature of the auction. Floor prices come indeed in two flavors:
\begin{itemize}
\item Hard floors (or reserve prices): the bid of the ad trader is only taken into account if it is higher than the floor price level $\phi$ set by the publisher (the supply side). When a hard floor level $\phi$ is set, everything works as if a ``ghost player'' was always bidding~$\phi$. In particular, the cumulative distribution function $F$ may be discontinuous, with a jump at price $\phi$.\footnote{It is noteworthy that different publishers may set different hard floors. In other words, there may be several jumps in $F$.}
\item Soft floors: if the best bid is below the soft floor level $\phi$, then the winner pays its own bid, and not the second price; otherwise, the price paid is the maximum between the second price and $\phi$. In other words, the auction is no more of the Vickrey type: it is a second-price auction with a hard floor for large bids and it becomes a first-price auction for small bids.\\
\end{itemize}

Our model can be generalized to tackle the case of hard floors by generalizing our approach to the case of a discontinuous cumulative distribution function $F$ -- in that case $f$ can be regarded as a distribution. The result is still that the budget should be spent evenly, but the optimal bidding strategy is more cumbersome to write -- and the exact expression has no theoretical interest. The case of soft floors can also be addressed, but it is really cumbersome. In particular, in the case of a uniform soft floor $\phi$ for all the auctions, Eq.~(\ref{m1:HJB2}) is replaced by an equation of the form

$$
-\partial_t v(t,S) - \lambda \inf\left\{\inf_{b \in [0,\phi]} F(b) (v(t,S-b) - v(t,S)-1),\right.$$
$$\left. \inf_{b>\phi}  \int_0^b f(p) (v(t,S-\max(p,\phi)) - v(t,S)-1) dp  \right\} = 0,$$
with terminal condition $v(T,S) = K \min\left(S,0\right)^2$.\\

In practice, it is also noteworthy that our bid may impact the strategic behavior of other users and in particular the behavior of the publishers who may set \textit{dynamic floor prices}. From a modeling point of view, it means that the function $f$ is impacted by the strategy itself, and it opens a vast field where (mean-field) game theory could be very useful.

\subsubsection{Nonlinear KPIs}

It is customary in the advertising industry to work with different metrics or KPIs. Practitioners usually aim indeed at minimizing the CPM, the CPC (cost per click) or the CPA (cost per acquisition).\\

With the notations of our model, the CPM at time $T$ is naturally defined by:
$$\text{CPM}_T = \frac{\bar{S} - S_T}{I_T}.$$
If the total spending is imposed to be $\bar{S}$, minimizing the expected value of $\text{CPM}_T$ boils down to minimizing $\mathbb{E}\left[\frac 1{I_T}\right]$, and not maximizing $\mathbb{E}[I_T]$ -- or, equivalently, minimizing $\mathbb{E}[-I_T]$, as in the model of Section 2.\\

It is noteworthy that our model somehow ignores the risk-aversion effect induced by the nonlinearity of traditional KPIs, but this is only a side-effect of convexity, as nonlinear KPIs have not been built to capture any form of risk aversion. In fact, the real issue is that the change of variables (from $u$ to $v$) at the heart of the dimensionality reduction used in our model does not extend to nonlinear KPIs. Nevertheless, from a mathematical perspective,  our approach based on the dynamic programming principle (and HJB equations \eqref{m1:HJB}) remains valid for any KPI.\footnote{The KPI only impacts the terminal condition on $u$.}\\

The above analysis for the CPM is also valid in the case of the CPC and the CPA. In the extension of the model involving conversions -- which may be regarded as clicks or acquisitions -- we considered a linear function of the number of conversions $\sum_{j=1}^J \delta^j C^j_T$, whereas practitioners would rather consider nonlinear KPIs related to the average amount paid to obtain the different types of conversion.\\

Overall, we think that linear KPIs are as relevant as the traditional nonlinear ones currently used by marketers (when the total budget is fixed), and that linear KPIs should be preferred, from a mathematical perspective, for solving the problems faced by ad trading desks.

\section*{Conclusion}

In this research paper, we have addressed several problems faced by media trading desks buying ad inventory through real-time auctions. These problems are new for the community of applied mathematicians, as the industry of programmatic advertising is itself quite recent. However, many ideas coming from the works of applied mathematicians on algorithmic trading in quantitative finance are inspiring for developing new models and subsequently lighting the way to innovation in a new and fast-growing field which is in demand for mathematical methods (as Finance was 25 years ago).\\

Our contribution is multifold. First of all, from a modeling perspective, we model auctions arriving at random times by marked Poisson processes, where the jumps of the Poisson processes stand for the occurrence of auction requests and the marks for several variables such as the best bid proposed by other participants, and the occurrence of a conversion. This approach makes it possible to use the dynamic programming principle and derive a simple characterization of the optimal bidding strategy -- through a Hamilton-Jacobi-Bellman equation. Furthermore, it allows to consider multiple sources and types of inventory in a very simple way.\\

Secondly, by considering linear objective criterions (or KPIs), we manage to reduce the dimensionality of the problem. In particular, whatever the initial number of inventory sources and types, and whether or not we consider conversions, we show that the problem always boils down to solving a Hamilton-Jacobi-Bellman PDE in dimension 2 (one time dimension and one spatial dimension).\\

Thirdly, by considering a fluid-limit approximation of the problem, we obtain almost-closed-form solutions for the optimal bidding strategy. Moreover, the results in the fluid-limit approximation allow to characterize the optimal strategy not only in terms of the optimal bids, but in terms of an optimal scheduling for the budget that remains to be spent. The latter simplifies the implementation, as the optimal bidding strategy can be dynamically approximated by a feedback-control tracking mechanism.\\

Eventually, the modeling approach we have proposed in this article opens the door to future research where the optimal control of the bidding strategy is coupled with on-line learning (treated in a companion article -- see \cite{fgllearning}), where the different participants in the auctions -- including the publisher -- could adopt strategic behaviors, where the information about the outcomes of auctions is subject to an important latency, etc.\\

\section*{Appendix: A simple model in discrete time}

In this appendix, we present a simplified and discrete-time version of the model introduced in Section 2 for providing the readers -- especially those who are not used with continuous-time models -- with intuition about the optimization approach used throughout the paper.\\

Let us introduce a probability space $(\Omega, \mathcal{F}, \mathbb{P})$ equipped with a discrete filtration $(\mathcal{F}_n)_{n \in \mathbb{N}}$ satisfying the usual conditions. The variables we define in this subsection are considered over this filtered probability space.\\

Let us consider a bidding system participating in a sequence of $N$ Vickrey auctions. We assume that the state of the system is described by three variables: (i) the number $n$ of auction requests received by the ad trading desk, (ii) the number $I_n$ of impressions purchased after the $n^{\text{th}}$ auction (we assume that $I_0 = 0$), and (iii) the remaining budget $S_n$ after the $n^{\text{th}}$ auction -- we assume that $S_0 = \bar{S}$ is the maximal total amount to be spent. The goal of the algorithm is to maximize an objective function of the form $\mathbb{E}\left[g(N,I_N,S_N)\right]$, for some function $g$. In particular, $g(N,I_N,S_N)$ can be equal to $I_N$ if we want to maximize the number of impressions.\\

For each auction request, the algorithm chooses a bid level. For the $n^{\text{th}}$ auction, we denote here by $b_n$ the bid sent to the auction server ($b_n$ has to be $\mathcal{F}_{n-1}$-measurable). When the $n^{\text{th}}$ auction occurs, a random variable $p_n$ is drawn with the probability density function $f$. This random variable represents the price to beat.\\

The resulting dynamics of the system is the following:
\begin{itemize}
\item If $b_n > p_n$, then the algorithm wins the auction, the price paid is $p_n$, and the system evolves from the state $(n-1,I,S)$ to the state $(n,I+1,S-p_n)$.
\item If $b_n \leq p_n$, then the algorithm does not win the auction and the system evolves from the state $(n-1,I,S)$ to the state $(n,I,S)$.
\end{itemize}

Let us define the \textit{value function}:
$$u(n,I,S) = \max_{(b_k)_{k>n} \in \mathcal{A}_n} \mathbb{E}\left[g(N,I_N,S_N)|\mathcal{F}_n\right],$$
where $(b_k)_{k>n} \in \mathcal{A}_n$ if and only if $(b_k)_{k>n}$ is a predictable process such that $S_N \ge 0$, almost surely.\\

In order to solve this optimization problem, we simply use the dynamic programming principle which yields the following Bellman equation:
\begin{equation*}
u(n-1,I,S) = \max_{b_{n} \in [0,S]}\mathbb{E}\left[u(n,I+1,S-p_n)\mathbf{1}_{b_n > p_n} + u(n,I,S)\mathbf{1}_{b_n\leq p_n}\right]. \qquad \text{(A.1)}
\end{equation*}

Eq. (A.1) can also be written as the following difference equation
$$ u(n,I,S) - u(n-1,I,S)  + \max_{b_n \in [0,S]}\mathbb{E}\left[(u(n,I+1,S-p_n) - u(n,I,S))\mathbf{1}_{b_n > p_n}\right] = 0,$$
\emph{i.e.}:
\begin{equation*}
u(n,I,S) - u(n-1,I,S) + \max_{b_n \in [0,S]}\int_0^{b_n}(u(n,I+1,S-p) - u(n,I,S)) f(p) dp = 0. \qquad \text{(A.2)}
\end{equation*}

From the terminal condition $u(N,I,S) = g(N,I,S)$, the value function $u$ can easily be approximated numerically by backward induction on a grid.\\

The optimal bid for the $n^{\text{th}}$ auction is given by the optimality condition $$u(n,I_{n-1}+1,S_{n-1}-b^*_n) = u(n,I_{n-1},S_{n-1})$$ or by $b^*_n = S_{n-1}$ if $u(n,I_{n-1}+1,0) > u(n,I_{n-1},S_{n-1})$. In particular, when the constraint is not binding, the optimal bid is consistent with the fact that, in the case of Vickrey auctions, participants have the incentive to bid their true valuation for the item.\\

It is noteworthy that a similar change of variables as in the continuous-time model of Section 2 can be used in the special case where $g(N,I,S) = I$. In that case, if we write $u(n,I,S) = I + v(n,S)$, then the Bellman equation (A.2) can be simplified into:
\begin{equation*}
v(n,S) - v(n-1,S) + \max_{b_n \in [0,S]}\int_0^{b_n}(v(n,S-p) - v(n,S) + 1) f(p) dp = 0,
\end{equation*}
and the terminal condition is $v(N,S) = 0$.\\

This model in discrete-time is useful to understand the general modeling framework we use, but it is limited. First, in practice, auctions arrive at random times and we do not know in advance how many auction requests the algorithm will receive. Moreover, for problems where different sources of auction requests have to be treated in parallel, the above discrete modeling approach is not convenient. Continuous-time models, where the (random) occurrence of auction requests are modeled by the jumps of a Poisson process, are definitely more realistic and flexible.\\

\bibliographystyle{plain}

\end{document}